\documentclass[reqno]{amsart}
\usepackage{amsthm}
\usepackage{amssymb,amsmath}
\usepackage{a4wide}
\usepackage{graphicx}
\usepackage{nicefrac}

\newtheorem{thm}{Theorem}
\newtheorem{lem}{Lemma}[section]
\newtheorem{cor}[lem]{Corollary}
\newtheorem{prop}[lem]{Proposition}

\newcommand{\mc}[1]{\ensuremath{\mathcal{#1}}}
\renewcommand{\d}[1]{\ensuremath{#1_k^{(d)}}}

\newcommand{\rem}{\medskip\noindent{\bf Remark. }}
\newcommand{\Ord}[1]{{\mathcal O}\left(#1\right)}

\newcommand{\E}{{\bf E}}
\newcommand{\bth}{\begin{thm}}
\renewcommand{\eth}{\end{thm}}
\newcommand{\bl}{\begin{lem}}
\newcommand{\el}{\end{lem}}
\newcommand{\bpf}{\begin{proof}}
\newcommand{\epf}{\end{proof}}
\newcommand{\eps}{\varepsilon}
\renewcommand{\(}{\left(}
\renewcommand{\)}{\right)}
\newcommand{\diff}[1]{\frac{\partial}{\partial {#1}}}
\newcommand{\U}[1]{\left[ #1 \right]_{u=1}}
\newcommand{\btr}[1]{\left| #1 \right|}
\newcommand{\pdiff}[3]{\frac{\partial^{#3}#1}{\partial #2^{#3}}}

\begin{document}

\title{The degree profile of P\'olya trees}
\thanks{This research has been
supported by {\em FWF (Austrian Science Foundation),
National Research Network S9600, grant S9604}.}

\author{Bernhard Gittenberger \and Veronika Kraus}
\thanks{Institute of Discrete Mathematics and Geometry, TU Wien,
Wiedner Hauptstr. 8-10/104, A-1040 Wien, Austria.}

\date{\today}

\begin{abstract}
We investigate the profile of random P\'olya trees of size $n$ when only nodes of degree $d$ are
counted in each level. It is shown that, as in the case where all nodes contribute to the profile,
the suitably normalized profile process converges weakly to a Brownian excursion local time.
Moreover, we investigate the joint distribution of the number of nodes of degree $d_1$ and $d_2$
in the levels of the tree.
\end{abstract}

\maketitle

\section{Introduction}

Consider the size of level $k$ in a rooted tree, i.e., the number of nodes at distance $k$ 
from the root. The sequence of level sizes of the tree is commonly called the profile of the tree.
First investigations on the profile of random trees of given size (i.e., their number of nodes)
seem to go back to Stepanov \cite{Stepanov69}. A first distributional result as well as the
relation to Brownian excursion local time was achieved by Kolchin \cite{Kol77} for the family
trees of a Galton-Watson branching process conditioned on the total progeny. Different 
representations of the same result have been obtained later by several authors by approaching the
problem from very different directions (random walks, queuing theory, general theory of stochastic
processes, random trees), see \cite{GS79, Ho82, Ta91,Ta91b}.

The relation between trees and diffusion processes has been studied by probabilistic methods as
well. In this context, first and foremost the seminal papers of Aldous \cite{Al91, Al91b, Al93} in
which he developped the theory of continuum trees should be mentioned. Here, a metric space called
continuum tree is identified as the limit of several classes of random trees with respect to the
Gromov-Hausdorff topology. This theory was further elaborated by Marckert and Miermont
\cite{MaMi09} and Haas and Miermont \cite{HaMi10}. These approaches tell us a lot about the
behaviour of large random trees. They imply limit theorems for global functionals. But they do
neither provide limit theorems for local functionals like the profile of trees nor always
convergence of moments. 

A different approach to the relation between trees and diffusion process 
was pursued in \cite{Pi99}. For a
general overview on the relation between stochastic processes of combinatorial and therefore
discrete origin and their continuous counterparts can be found in \cite{Pi06}.

Not only the relation between trees and diffusion processes attracted much attention, but the
processes themselves are of interest in their own right. A thorough overview on Brownian local
times and related processes can be found in \cite{RY91}. Explicit representations for the moments
and the density of the one-dimensional projections of the local time of a Brownian excursion and
related processes have been derived by Tak{\'a}cs \cite{Ta95, Ta95b, Ta99}. 
Multi-dimensional analogues
can be found in \cite{GL99, GL00}. For results on density representations for related processes
such as occupation times we refer to \cite{DrGi97, Ja97, Ho99}.

The profile of random trees has recently attracted the attention of numerous authors. A survey on
the theory of random trees in general and the profile in particular can be found in \cite{MD09}. 
Roughly
speaking, the tree classes which have been studied can be divided into trees of height $\asymp\log n$ ($\asymp$ meaning the order of magnitude) and trees of height $\asymp\sqrt n$ where $n$ is the size of the tree. Trees of logarithmic height are for instance binary search trees and recursive trees and variations or generalizations thereof. The profile of binary search trees and variations has been analyzed by Chauvin et al. \cite{CDJ01, CKMR05} and generalized by Drmota et al. \cite{DJN08} and 
Schopp \cite{Sch10}. Hwang et al. studied the profile of tree classes related to recursive trees,
see \cite{DrHw05, Hwa07, FHN06}. The maximum of the profile, commonly called the width of a
tree, in logarithmic trees was analyzed by Devroye and Hwang \cite{DeHw06}. A related structure
appearing in the theory of data structures are so-called tries. Their profile was examined in
\cite{De02, Ni05, PHNS08}. 

Trees of height $\asymp\sqrt n$ are for instance trees obtained from conditioned Galton-Watson
branching processes or P\'olya trees, i.e., rooted trees in the pure graph theoretical sense. 
The analysis of the profile of Galton-Watson trees was done in \cite{DrGi97}. The start of the
profile, i.e., the behaviour close to the root, for trees as well as forests was investigated in
\cite{G98,G02}. Binary P\'olya trees have been studied by Broutin and Flajolet
\cite{BrFl08,BrFl11} and general P\'olya trees in \cite{DrGi}. 

The joint distribution of two level sizes was addressed explicitely by van der Hofstad et al.
\cite{HHM02} and in \cite{GL99} for Galton-Watson trees. Note that this question also arises
implicitely when proving a functional limit theorem by showing convergence of the
finite-dimensional projections as well as tightness of the profile, albeit the proofs of tightness
utilizes the moments of the difference of the two level sizes. 

Recently, patterns in random trees were investigated as well. The questions considered here are the
occurence of certain given trees as substructure of a large tree as well as the number of such
occurences. The easiest pattern is the star graph. This amounts to counting the number of nodes of
a given degree in random trees. First investigations in this directions were performed by Robinson
and Schwenk  \cite{RoSch}. In \cite{DrGi99} it was shown for several tree classes (certain classes
of Galton-Watson trees as well as P\'olya trees) that the number of nodes of given degree is
asymptotically normally distributed, as long as the degree fixed. A phase transition occurs if the
given degree grows with the tree size, see \cite{MM91} for Galton-Watson trees and \cite{Gi06}
for P\'olya trees. The analogues for general patterns instead of star graphs was carried out by
Chyzak et al. \cite{CDKK08}. 

Note that all results on patterns in trees tell us something about the number of occurences of a
given pattern, but nothing about their location within the large tree. For Galton-Watson trees
this question was settled by Drmota \cite{Dr97} for star graphs. The same question for P\'olya
trees is addressed in this paper.

\subsection*{Plan of the paper} In the next section we will recall some basic results on P\'olya
trees and present the main results afterwards. The first result is Theorem~\ref{thm:1} which
states that the d-profile (number of nodes of degree $d$ at fixed distance from the root) view as
a stochastic process weakly converges to Brownian excursion local time. To prove this theorem we
will split it into two partial results, the convergence of the finite-dimensional projections of
the profile (Theorem~\ref{th:3}) as well as the tightness of the profile
(Theorem~\ref{thm:tightness}). To examine the joint behaviour of two different patterns we derive
the covariance and the correlation (Proposition~\ref{prop:covariance} and Theorem~\ref{thm:cor}).
The results on covariance and correlation exhibit a surprisingly regular structure of the limiting
object. We show that the correlation coefficient tends to 1 and derive the speed of convergence as
well.

The proof of these theorems will be carried out by means of generating functions. This will be
described in Section~\ref{sec:2} as well. This section ends with an introduction of the notation we
will use throughout the paper. 

To proceed we need a singularity analysis of the generating functions together with a kind of
transfer of the singular behaviour into the asymptotic behaviour of the coefficients in the sense
of \cite{FO90}. Section~\ref{sec:3} provides some \emph{a priori} estimates which are to be
refined later. In Section~\ref{sec:4} we first present the proof of the one-dimensional analogue
of Theorem~\ref{th:3} which is based on the refinement of the results in Section~\ref{sec:3}. The
rest of the section is devoted to the refined analysis. The generalization to multiple dimensions
is done in Section~\ref{sec:5}. In Section~\ref{sec:6} we show tightness. This is usually a very
technical matter (cf. the eight-page proof in \cite{DrGi}). Here we offer a considerably shorter
proof by showing a more general result using Fa\`a di Bruno's formula. The final section is
devoted to the joint behaviour of the numbers of nodes of two different degrees within one level.

\section{Results and Notation}\label{sec:2}

\subsection{Preliminaries}

A P\'olya tree is an unlabelled rooted tree and thus it can be viewed as a root with a set of
P\'olya trees attached to it. By the machinery of symbolic transfers described in \cite{FlSe},
we easily obtain that the generating function $y(x)$ of unlabelled rooted trees fulfils the 
functional equation
\begin{equation} \label{treefun}
y(x) = x\exp\(\sum\limits_{i \geq 1} \frac{y(x^i)}{i}\), 
\end{equation} 
a result going back to P\'olya \cite{Pol37} (cf also~\cite{PoRe87}) who also showed that $y(x)$ has exactly one
singularity $\rho$ on the circle of convergence and that $\rho \approx 0.3383219$. 
Around its singularity, $y(x)$ has the local expansion 
\begin{align}\label{exp}
y(x) = 1 - b\sqrt{\rho - x} + c(\rho-x) +d\sqrt{\rho-x}^3 + \cdots,
\end{align}
with $b\approx 2.6811266$, as Otter \cite{Ot} showed, and $y(\rho)=1$. From the expansion, asymptotic estimations for $y_n$
can be derived by transfer lemmas (cf. \cite{FlSe}):
\begin{align}\label{eq:yn}
y_n \sim \frac{b \sqrt{\rho}}{2 \sqrt{\pi}} \frac{1}{n^{\frac{3}{2}}\rho^n}
\end{align}

\subsection{Main results}

We define by $L_n^{(d)}(k)$ the number of nodes of degree $d$ at distance $k$ from the root in a
randomly chosen unlabelled rooted tree of size $n$. By linear interpolation, we create a continuous stochastic process 
$L_n^{(d)}(t) = (\lfloor t \rfloor + 1 -t)L_n^{(d)}(\lfloor t\rfloor) + (t-\lfloor t \rfloor)L_n^{(d)}(\lfloor t \rfloor + 1), \quad t \geq 0
$

\begin{thm}\label{thm:1}
Let 
\begin{align*}
l_n^{(d)}(t) = \frac{1}{\sqrt{n}} L_n^{(d)}(t \sqrt{n})
\end{align*}
and $l(t)$ denote the local time of a standard Brownian excursion. Then $l_n^{(d)}(t)$ converges 
weakly to the local time of a Brownian excursion, i.e., we have 
\begin{equation} \label{mainres}
(l_n^{(d)}(t))_{t \geq 0} \stackrel{w}{\to} \frac{C_d \rho^d}{\sqrt{2 \rho}b} \cdot l\left( \frac{b \sqrt{\rho}}{2 \sqrt{2}}t \right)_{t\geq0},
\end{equation} 
where $C_d=C+\Ord{d\rho^d}$ with $C=\exp\left(\sum\limits_{i \geq 1} \frac{1}{i}\left(\frac{y(\rho^i)}{\rho^i}-1\right)\right)\approx 7.7581604\dots$. 
\end{thm}

\rem In \cite{DrGi} it is shown that the general profile of an unlabelled rooted random tree
converges to Brownian excursion local time, i.e.
\begin{align*}
(l_n(t))_{t\ge 0} \to \(\frac{b \sqrt{\rho}}{2 \sqrt{2}} l \big( \frac{b \sqrt{\rho}}{2 \sqrt{2}}
t\big)\)_{t\ge 0}
\end{align*}
The normalising constant in Theorem \ref{thm:1} equals $\mu_d \frac{b \sqrt{\rho}}{2 \sqrt{2}}$,
where $\mu_d n$ is asymptotically equal to the expected value of nodes of degree $d$ in trees of
size $n$, with $\mu_d = \frac{2 C_d}{b^2\rho} \rho^d$, see for example \cite{Da} or \cite{DrGi99}. 
To prove the above statement, weak convergence of the finite dimensional distributions and tightness have to be shown:

\begin{thm}\label{th:3}
For any choice of fixed numbers $t_1, \ldots, t_m$ and for large $d$ 
\begin{align*}
(l_n^{(d)}(t_1), \ldots, l_n^{(d)}(t_m)) \stackrel{w}{\to}  \frac{C_d \rho^d}{\sqrt{2 \rho} b}
l\left(\frac{b \sqrt{\rho}}{2 \sqrt{2}} t_1, \ldots, \frac{b \sqrt{\rho}}{2 \sqrt{2}}t_m \right)
\end{align*}
as $n \to \infty$.
\end{thm}

\rem
We will show this theorem by proving the convergence of the corresponding characteristic
functions. It is well known (cf. \cite{CoHo81}) 
that the characteristic function of $\frac{C_d \rho^d}{\sqrt{2 \rho} b}
l\(\frac{b \sqrt{\rho}}{2 \sqrt{2}} t\)$ is 
\begin{equation} \label{local_time_char_fun}
\psi(t)=1+ \frac{C_d \rho^d}{ib\sqrt{\rho\pi} }\int_\gamma \frac{t\sqrt{-x} \exp\(-\frac{\kappa
b}{2\sqrt{-\rho x}}-x\)}{\sqrt{-x} \exp\(\frac{\kappa b}{2\sqrt{-\rho x}}\) -
\frac{C_d\rho^d t}{b\sqrt{\rho}} \sinh\(\frac{\kappa b}{2\sqrt{-\rho x}}\)} \, dx
\end{equation} 
where $\gamma$ is a contour going from $+\infty$ back to $+\infty$ while encircling the origin
clockwise.

\medskip
A sequence of stochastic processes might not converge even if the sequence of their images with
respect to every finite-dimensional projection does. Roughly speaking, in order to guarantee
convergence in the sense of stochastic processes (i.e., when constructing a sequence by 
applying an arbitrary continuous bounded functional to the corresponding probability measures, 
this sequence must converge) the sample paths of the processes must not fluctuate too wildly.  
Tightness is a technical property of stochastic processes which guarantees this. The next
theorem states a technical condition for the profile process which implies tightness (cf.
\cite{Bi} and \cite{KS} for the general theory).

\begin{thm} \label{thm:tightness}
There exists a constant $c>0$ such that all integers $r,h,n$ the inequality 
\begin{equation} \label{tightness}
\E\left( L_n(r)-L_n(r+h)\right)^4 \le c\, h^2 n
\end{equation} 
holds.
\end{thm}

\rem According to \cite[Theorem 12.3]{Bi} the inequality 
$$
{\bf E}\, |L_n(r) - L_n(r+h)|^\alpha = \Ord{ h^\beta (\sqrt n)^{\alpha-\beta}}
$$
implies tightness of the process $l_n(t)$ if $\alpha>0$ and $\beta>1$. In the theorem above we
have $\alpha=4$ and $\beta=2$ and thus $l_n(t)$ is tight. We remark here that in 
\cite[remark on p.2050]{DrGi} the authors erroneously stated the bound $\Ord{ (h \sqrt n)^\beta}$.

\medskip\noindent
In order to examine the dependence of the numbers of nodes for two different degrees, say $d_1$ and $d_2$ (at the same
level $k$), we will compute the covariance and the correlation. 

\begin{prop}\label{prop:covariance}
 The covariance $\mathbb{C}\mathrm{ov}(X_n^{(d_1)}(k),X_n^{(d_2)}(k))$ of random variables
$X_n^{(d_1)}(k)$ and $X_n^{(d_2)}(k)$ counting vertices of degrees $d_1$ and $d_2$, with $d_1\neq
d_2$ fixed, at level $k=\kappa{n}$ in a random P\'olya tree of size $n$ is asymptotically given by
\begin{equation}\label{eq:covariance}
 \mathbb{C}\mathrm{ov}(X_n^{(d_1)}(k),X_n^{(d_2)}(k))=
C_{d_1}C_{d_2}\rho^{d_1+d_2}n\left(\frac{2}{b^2\rho}\left(e^{-\frac{\kappa^2b^2\rho}{4}}
+e^{-\kappa^2b^2\rho}\right)-\kappa^2e^{-\frac{\kappa^2b^2\rho}{2}}\right)\(1+O\(\frac1{\sqrt
n}\)\),
\end{equation}
as $n$ tends to infinity.
\end{prop}

\begin{thm}\label{thm:cor}
 Let $X_n^{(d_1)}(k)$ and $X_n^{(d_2)}(k)$ be the random variables counting the number of
vertices of degree $d_1$ and $d_2$, respectively, on a level $k=\kappa\sqrt{n}$ in a P\'olya tree
of size $n$. Then the correlation coefficient \index{correlation coefficient} is asymptotically
equal to $1$ as $n$ tends to infinity. The speed of convergence is of order $\nicefrac 1{\sqrt n}$.
\end{thm}

\subsection{Description of the problem with generating functions}
Proving Theorem \ref{th:3}, we will start with the one-dimensional case and then extend results to
multiple dimensions. Therefore, we introduce generating functions $y_k^{(d)}(x,u)$, which
represent trees where all nodes of degree $d$ on level $k$ are marked and counted by $u$. Note
that we consider planted trees instead of 'ordinary' rooted trees, that is, we assume that the
root node is adjacent to an additional node which is not counted. This assumption does not alter
the tree structure, but allows us to treat the root vertex like a normal vertex, that is, a root
of degree $d$ has in-degree $1$ and out-degree $d-1$.
Refining the decomposition of trees along their root, the $y_k^{(d)}(x,u)$ can be defined
recursively: 
\begin{align}\label{eq:rec}
y_0^{(d)} (x,u) &= y(x) + (u-1)x Z_{d-1}(y(x),y(x^2),\ldots, y(x^{d-1})) \nonumber\\
y_{k+1}^{(d)} (x,u)&= x \exp\(\sum_{i \geq 1}{\frac{y_k^{(d)}(x^i,u^i)}{i}}\),
\end{align}
where $Z_{d}(s_1,s_2,\ldots,s_d)$ is the cycle index of the symmetric group $\mathfrak{S}_d$ on
$d$ elements, given by
\[
\frac{1}{|\mathfrak{S}_d|}\sum_{\pi \in \mathfrak{S}_d}\prod_{i=1}^ds_i^{\lambda_i},
\] 
where $\lambda_i$ is the number of cycles of length $i$ in the permutation $\pi$. 

Examining two levels $k$ and $k+h$ simultaneously, we use the generating function $y_{k,h}^{(d)}(x,u_1,u_2)$ where all nodes of degree $d$ on level $k$ are marked by $u_1$ and nodes of degree $d$ on
level $k+h$ are marked by $u_2$. As before, $x$ marks the total size of the tree. We get the
recursive relation 
\begin{align}\label{eq:rec2dim}
{y}_{0,h}^{(d)}(x,u_1,u_2) &=y_h^{(d)}(x,u_2) + (u_1-1)x Z_{d-1}(y_h^{(d)}(x,u_2), \ldots, y_h^{(d)}(x^{d-1},u_2^{d-1}))\nonumber\\
{y}_{k+1,h}^{(d)}(x,u_1,u_2)&=x \exp \left( \sum\limits_{i \geq 1}\frac{{y}_{k,h}^{(d)}(x^i,u_1^i,u_2^i)}{i}\right)
\end{align}

In general, observing levels $k_1$,$k_2=k_1+h_1$,\ldots, $k_m=k_{m-1}+h_{m-1}$, we get:
\begin{align*}
{y}&_{0,h_1,\ldots,h_{m-1}}^{(d)}(x,u_1, \ldots, u_m) =y_{h_1, \ldots h_{m-1}}^{(d)}(x,u_2,\ldots, u_m) \\
&+ (u-1)x Z_{d-1}(y_{h_1, \ldots h_{m-1}}^{(d)}(x,u_2,\ldots, u_m), \ldots, y_{h_1, \ldots h_{m-1}}^{(d)}(x^{d-1},u_2^{d-1},\ldots, u_m^{d-1}))\\
{y}&_{k+1,h_1,\ldots,h_{m-1}}^{(d)}(x,u_1,\ldots, u_m)=x \exp \left( \sum\limits_{i \geq 1}\frac{{y}_{k,h_1,\ldots, h_{m-1}}^{(d)}(x^i,u_1^i,\ldots, u_m^i)}{i}\right)
\end{align*}
These functions are related to the process $L_n^{(d)}(t)$ by
\begin{align*}
\mathbb{P}(L_n^{(d)}(k)=\ell_1,L_n^{(d)}(k+h_1)=\ell_2,\ldots,L_n^{(d)}(k+\sum h_i)=\ell_m)\\
=\frac{[x^nu_1^{\ell_1}u_2^{\ell_2}\cdots u_m^{\ell_m}]{y}_{0,h_1,\ldots,h_{m-1}}^{(d)}(x,u_1, \ldots, u_m)}{[x^n]y(x)}
\end{align*}
For the computation of the covariance of the numbers of nodes of degrees $d_1$ and $d_2$ 
we will utilize the functions 
\begin{align} \label{eq:rec_var}
y_0^{(d_1,d_2)} (x,u,v) &= y(x) + (u-1)x Z_{d_1-1}(y(x),y(x^2),\ldots, y(x^{d_1-1})) \nonumber \\ 
&\qquad+(v-1)x
Z_{d_2-1}(y(x),y(x^2),\ldots, y(x^{d_2-1})) \nonumber \\
y_{k+1}^{(d_1,d_2)} (x,u,v)&= x \exp\(\sum_{i \geq 1}{\frac{y_k^{(d_1,d_2)}(x^i,u^i,v^i)}{i}}\).
\end{align}

\subsection{Notations}

For proving the theorems we will carry out a singularity analysis of the generating functions.
All generating functions are in some sense close to the tree function $y(x)$ from \eqref{treefun}. 
Therefore we will use the differences to $y(x)$ and related functions: 
Set 
\begin{align}
w_k^{(d)}(x,u)&=y_k^{(d)}(x,u)-y(x) & w_k^{(d_1,d_2)}(x,u,v)&=y_k^{(d_1,d_2)}(x,u,v)-y(x)\nonumber\\
\Sigma_k^{(d)}(x,u)&=\sum_{i \geq 2}\frac{w_k^{(d)}(x^i,u^i)}{i} & \Sigma_k^{(d_1,d_2)}(x,u,v)&=
\sum_{i
\geq 2}\frac{w_k^{(d_1,d_2)}(x^i,u^i,v^i)}{i}\nonumber\\
\gamma_k^{(d)}(x,u)&=\frac{\partial}{\partial u}y_k^{(d)}(x,u) &
\tilde \gamma_k^{(d_1,d_2)}(x,u,v)&=\frac{\partial^2}{\partial u\partial v}y_k^{(d_1,d_2)}(x,u,v)\label{eq:tildeg}\\
\gamma_k^{(d)[2]}(x,u)&=\frac{\partial^2}{\partial u^2}y_k^{(d)}(x,u)\nonumber
\end{align}

We further introduce some domains, depicted in Figure~\ref{fig:regions}:
\begin{align}
 \Delta=\Delta(\eta,\theta)&=\{z\in\mathbb{C}\big||z|<\rho+\eta,|\arg(z-\rho)|>\theta\},\\
 \Delta_\epsilon=
\Delta_\epsilon(\theta)&=\{z\in\mathbb{C}\big||z-\rho|<\epsilon,|\arg(z-\rho)|>\theta\},\\
 \Theta=\Theta(\eta)&=\{z\in\mathbb{C}\big||z|<\rho+\eta, |\arg(z-\rho)|\neq0\},\\
 \Xi_k=\Xi_k(\tilde{\eta})&=\{z\in\mathbb{C}\big||v|\leq1,k|v-1|\leq\tilde{\eta}\},
\end{align}
with $\epsilon,\eta,\tilde{\eta}>0$ and $0<\theta<\frac{\pi}{2}$.
\begin{figure}[ht]
\centering
\includegraphics[width=0.8\textwidth]{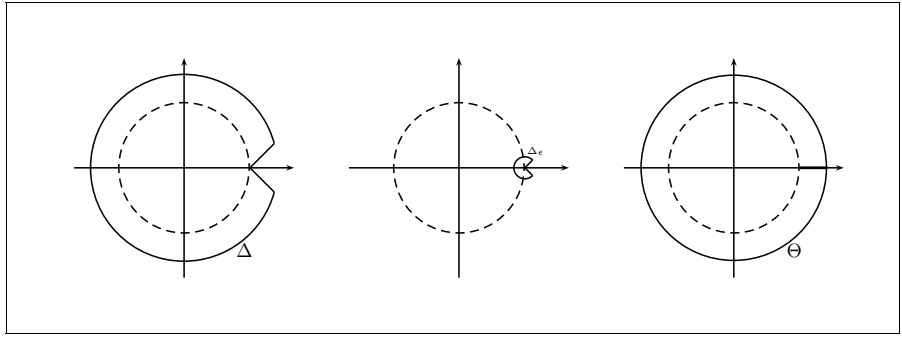}
\caption{The regions used for the proofs}\label{fig:regions}
\end{figure}

In all the proofs in the subsequent sections we will assume (even without explicitely mentioning)
that $\eta,\theta,\epsilon$ are sufficiently small for all arguments to be valid.

\section{The local behaviour of $y_k^{(d)}(x,u)$ and $y_k^{(d_1,d_2)}(x,u,v)$ -- A priori
bounds}\label{sec:3}

In order to analyze the local behaviour of the generating functions we will first derive \emph{a
priori} estimates for $w_k^{(d)}$, $w_k^{(d_1,d_2)}$ and the related functions which will be
used frequently in the sequel to derive the needed refinements. 

\begin{lem}
Let $|x|\leq\rho^2+\varepsilon$ for sufficiently small $\varepsilon$ and $|u|\leq1$. Then there
exists a constant $L$ with $0<L<1$ and a positive constant $D$ such that \[|\d{w}(x,u)| \leq D |u-1|\cdot |x|^d \cdot L^k\]
\end{lem}

\begin{proof} We will only provide a short sketch, since the proof is similar to that of
\cite[Lemma~2]{DrGi}. 

 For $k=0$ we have \[|w_0^{(d)}|(x,u)|=|u-1|\cdot|x|\cdot\underbrace{|Z_{d-1}(y(x),y(x^2),\ldots)|}_{\mathcal{O}(|y(x)|^{d-1})=\mathcal{O}(|x|^{d-1})} \leq |u-1| \cdot D\cdot|x|^d\]

The result for general $\d{w}(x,u)$ follows by induction. Starting with the recurrence relation
$$
w_{k+1}^{(d)}(x,u)=y(x)\(\exp\(\d{w}(x,u)+\sum_{i\ge 2}\frac{\d{w}(x^i,u^i)}i\)-1\)
$$
we use the trivial estimate $|w_k(x,u)|\le 2y(|x|)$ which is valid for $|x|\le \rho$ and $|u|\le
1$, the convexity of $\nicefrac{y(x)}{x}$ on the positive reals, and some elementary estimates for
$e^x$. For the precise details see \cite{DrGi}. 
\end{proof}

\begin{cor}\label{cor:sigma}
For $|u|\leq 1$ and $|x|\leq \rho + \varepsilon$ there is a positive constant $\tilde{C}$ such that (for all $k \geq 0, d\geq 1$) \[|\Sigma_k^{(d)}(x,u)|\leq \tilde{C}|u-1|L^k.\]
\end{cor}

\begin{proof}
Same as proof of Corollary 1 in \cite{DrGi}.
\end{proof}

\begin{cor}\label{cor:sumgamma}
 Let $u \in \Xi_k$ and $x \in \Theta$. Then 
\[\sum_{i \geq 2}\gamma_k^{(d)}(x^i,u^i) = \mathcal{O}(L^k).\]
\end{cor}

\begin{proof} 
As $i \geq 2$ the functions $\gamma_k^{(d)}(x^i,u^i)$ are analytic in the whole region and  
$\Gamma_k^{(d)}(x,u):=\sum_{i\ge 2}\gamma_k^{(d)}(x^i,u^i) = \sum_{n,m}y_{nmk}^{(d)}x^ny^m$ with 
positive coefficients $y_{knm}^{(d)}$, we have $|\Gamma_k^{(d)}(x,u)| \leq
\Gamma_k^{(d)}(|x|,|u|)$ where the right-hand side is monotone in $|x|$ and $|u|$. 

Now let $x\ge 0$ and $0<u<1$. Using Taylor's theorem we get 
$$
\d\Sigma(x,u)=(u-1)\d\Gamma(x,1+\vartheta (u-1))\ge \d\Gamma(x,u).
$$
In view of Corollary~\ref{cor:sigma} this implies for all $x\in\Theta$ and $u\in \Xi_k$ the
estimate $|\d\Gamma(x,u)|\le \d\Gamma(|x|,|u|) \le CL^k$ for some positive constant $L<1$.
\end{proof}

In a similar fashion we obtain the analogous results for $w_k^{(d_1,d_2)}$: 

\begin{lem}
If $|x|\leq\rho^2+\varepsilon$ for sufficiently small $\varepsilon$, $|u|\leq1$, and $|v|\le 1$, 
then there 
exists a constant $L$ with $0<L<1$ and a positive constant $D$ such that
\[|w_k^{(d_1,d_2)}(x,u,v)| \leq D L^k (|u-1|\cdot |x|^{d_1}+|v-1|\cdot |x|^{d_2}) \]
\end{lem}

\begin{cor}
For $|u|\leq 1$, $|v|\le 1$ and $|x|\leq \rho + \varepsilon$ there is a
positive constant $\tilde{C}$ such that (for all $k \geq 0, d\geq 1$)
\[|\Sigma_k^{(d_1,d_2)}(x,u,v)|\leq
\tilde{C}(|u-1|+|v-1|)L^k.\]
\end{cor}

\begin{cor} 
Let $u \in \Xi_k$, $v \in \Xi_k$, and $x \in \Theta$. Then 
\begin{equation} \label{eq:tildegamma} 
\sum_{i \geq 2}\tilde \gamma_k^{(d_1,d_2)}(x^i,u^i,v^i) = \mathcal{O}(L^k). 
\end{equation}  
\end{cor}

\section{The one dimensional case}\label{sec:4}

Our main goal is to prove the following theorem from where the main result follows by integration.

\begin{thm}\label{th:exp}
Let $x = \rho (1+\frac{s}{n})$, $u=e^{\frac{it}{\sqrt{n}}}$, $k=\lfloor{\kappa \sqrt{n}}\rfloor$ and $d$ be a fixed integer. Moreover, assume that $|\arg s| \geq \vartheta > 0$ and, as $n \to \infty$, we have $s=\mathcal{O}(\log^2 n)$, whereas $\kappa$ and $t$ are fixed. Then, $w_k^{(d)}(x,u)$ admits the local representation 

\begin{align} \label{eq:wkd}
w_k^{(d)}(x,u) \sim \frac{C_d \rho^d}{\sqrt{n}} \cdot \frac{it\sqrt{-s}e^{-\frac{1}{2}\kappa b
\sqrt{-\rho s}}}{\sqrt{-s}e^{\frac{1}{2}\kappa b \sqrt{-\rho s}} - \frac{itC_d\rho^{d}}{b\sqrt{\rho}}\sinh(\frac{1}{2}\kappa b \sqrt{-\rho s})}
\end{align}
\end{thm}

\subsection*{The one-dimensional limiting distribution}
Let us first assume that Theorem~\ref{th:exp} holds. Then, to prove Theorem \ref{th:3} in one dimension, we need to determine the characteristic function
\begin{align}\label{eq:phi}
\phi_{k,n}^{(d)}(t) &= \frac{1}{y_n} [x^n] y_k^{(d)}(x,e^{\frac{it}{\sqrt{n}}}) \nonumber\\
&=\frac{1}{2 \pi i y_n} \int_\Gamma y_k^{(d)}(x,e^{\frac{it}{\sqrt{n}}}) \frac{dx}{x^{n+1}}
\end{align}
where the contour $\Gamma = \gamma \cup \Gamma'$ consists of the line 

\begin{align*}
\gamma = \{x = \rho(1 - \frac{1+i\tau}{n}) | -D \log^2 n \leq \tau \leq D \log^2 n\}
\end{align*}
with an arbitrarily chosen constant $D > 0$ and $\Gamma'$ is a circular arc centered at the origin
and closing the curve. The contribution of $\Gamma'$ is exponentially small since for $x \in
\Gamma'$ we have $\frac{1}{y_n}|x^{-(n+1)}| \mc{O}(n^{\frac{3}{2}} e^{-\log^2 n})$ on the one hand whereas
on the other hand $|y_k^{(d)}(x,e^{\frac{it}{\sqrt{n}}})|$ is bounded. 

If $x \in \gamma$ the local expansion \eqref{eq:wkd} is valid and thus, inserting into \eqref{eq:phi} leads to:

\begin{align*}
\lim_{n\to\infty}\phi_{k,n}^{(d)}(t) &= \lim_{n\to\infty} 
\frac{1}{2 \pi i y_n} \bigg[\int_{\Gamma'} w_k^{(d)}(x,u) \frac{dx}{x^{n+1}} + \underbrace{\int_{\Gamma'} y(x) \frac{dx}{x^{n+1}}}_{=2 \pi i y_n}\bigg]\\
&= 1+\lim_{n\to\infty}\frac{C_d \rho^{d+n} n \sqrt{2}}{b \sqrt{2 \rho \pi}} \int\limits_{1-i
\log^2 n}^{1+i \log^2 n} \frac{t \sqrt{-s} e^{(-\frac{\kappa b \sqrt{-\rho
s}}{2})}}{e^{(\frac{\kappa b \sqrt{-\rho s}}{2})} - \frac{it C_d \rho^d}{\sqrt{\rho} b} \sinh (\frac{\kappa b \sqrt{-\rho s}}{2})} \frac{1}{\rho^n n} e^{-s} ds \\
&= \psi(t)
\end{align*} 
where $\psi(t)$ is the function given by \eqref{local_time_char_fun}.\\

Now let us turn back to the proof of Theorem~\ref{th:exp}.

\subsection{The local behaviour of $y_k^{(d)}(x,u)$ - refined analysis}

Now we will refine the \emph{a priori} estimates of the previous section. First we show that the
first derivate $\gamma_k^{(d)}(x,1)$ is almost a power of $y(x)$. Afterwars we will derive estimates
for the second derivative and then obtain a power-like representation for $w_k^{(k)}(x,u)$.
Finally, utilizing the recurrence relation for $w_k^{(k)}(x,u)$ we will arrive at the desired
result \eqref{eq:wkd}.

\begin{lem}\label{l:g}
For $x \in \Theta$ (where $\eta>0$ is sufficiently small) the functions $\gamma_k^{(d)}(x)$ can be represented as 
\[\gamma_k^{(d)}(x):=\gamma_k^{(d)}(x,1)=C_k^{(d)}(x)y(x)^{k+d},\]
where the functions $C_k^{(d)}(x)$ are analytic and converge uniformly to an analytic limit
function $C^{(d)}(x)$ (for $x\in\Theta$) with convergence rate \[C_k^{(d)}(x) = C^{(d)}(x) +
\mathcal{O}(L^{k})\] for some $0<L<1$, and further $C^{(d)}(\rho) = C_d \rho^d$, where $C_d$ is
the constant given in \eqref{mainres}.
\end{lem}
\begin{proof}
We define the functions $C_k^{(d)}(x):=\frac{\gamma_k^{(d)}(x)}{y(x)^{k+d}}$.\\

We prove the analyticity of the functions $\gamma_k^{(d)}(x)$ by induction:
\[\gamma_0^{(d)}(x)=xZ_{d-1}(y(x),y(x^2),\ldots, y(x^{d-1}))=x\mathcal{O}(y(x)^{d-1})\] is analytic in $\Theta$ by previous arguments, and so is 
\begin{equation}\label{eq:C0}
C_0^{(d)}(x)=\frac{x\mathcal{O}(y(x)^{d-1})}{y(x)^d}=\mc{O}(1)
\end{equation}
The step of induction works like in \cite{DrGi}, as the $\gamma_k^{(d)}$ fulfill the same recursion as the $\gamma_k$: 
\begin{align}\label{eq:recgamma}
\gamma_{k+1}^{(d)}(x,u) &= \frac{\partial}{\partial u} xe^{\sum\limits_{i \geq 1} y_k^{(d)}(x^i,u^i)} \nonumber\\
&=x e^{\sum\limits_{i \geq 1}\frac{y_k^{(d)}(x^i,u^i)}{i}} \sum\limits_{i \geq 1}\frac{\partial}{\partial u}{y_k^{(d)}(x^i,u^i)}u^{i-1}\nonumber\\
&=y_{k+1}^{(d)}(x,u)\sum\limits_{i\geq 1} \gamma_k^{(d)}(x^i,u^i)u^{i-1},
\end{align}
and for $u=1$ \[\gamma_{k+1}^{(d)}(x)=y(x)\d{\gamma}(x) + y(x)\Gamma_k^{(d)}(x),\] with $\Gamma_k^{(d)}(x)=\sum\limits_{i\geq 2} \gamma_k^{(d)}(x^i)$, which is analytic for $|x|\leq\sqrt{\rho}$ and hence in $\Theta$. Applying the induction hypothesis, this proves the analyticity of $\d{\gamma}(x)$. 
Solving the recurrence, we obtain 
\begin{align*}\d{\gamma}=y(x)^{k}\gamma_0^{(d)}(x)+\sum_{\ell=0}^{k-1}y(x)^{k-\ell}\Gamma_{\ell}^{(d)}(x)
\end{align*}
and hence, the analyticity of $\d{\gamma}$ implies the analyticity of the functions $C_k^{(d)}(x)$ in $\Theta$.\\
We now have to show that the functions $(C_k^{(d)}(x))_{k \geq 0}$ have a uniform limit
$C^{(d)}(x)$ but this works analogously to \cite[Lemma~3]{DrGi}.  

Finally, note that 
\[
\sum_{k \geq 0}\gamma^{(d)}_k(x,1)=\sum_{k \geq 0} d_n^{(d)}x^n =
D^{(d)}(x),
\] 
where $d_n^{(d)}$ is the total number of vertices of degree $d$ in all trees of size
$n$, and $D^{(d)}(x)$ is the corresponding generating function, introduced in e.g. \cite{RoSch}. 
On the other hand, 
\[
\sum_{k \geq 0}\gamma^{(d)}_k(x,1)=\sum_{k\geq 0}(C^{(d)}(x)+\mathcal{O}(L^k))y(x)^{k}=\frac{C^{(d)}(x)y(x)^d}{1-y(x)}+\mathcal{O}(1)\] and therefore \[C^{(d)}(\rho) = \lim_{x \to \rho}\frac{(1-y(x))D^{(d)}(x)}{y(x)^d}.
\]
We know that 
$$
D^{(d)}(x)=\frac{y(x)\sum_{i\ge 2} D^{(d)}(x^i)+xZ_{d-1}(y(x),\dots,y(x^{d-1}))}{1-y(x)}
$$
(cf. \cite[Eq.~(36)]{RoSch} or \cite{Da}). Schwenk \cite[Lemma~4.1]{Sch77} computed the limit of the cycle index in the
numerator. In his proof he provides the speed of convergence as well. In fact,
\cite[Eq.~(32)]{Sch77} says that 
$$
\left|Z_d\(\frac{y(x)}x,\dots,\frac{y(x^d)}{x^d}\)-\exp\(\sum_{i=1}^d
\frac1i\(\frac{y(x)}x-1\)\)\right|\le x^{d+1} \exp\(\lambda \sum_{i=1}^d \frac 1i\)
$$
with $\lambda=\sup_{0\le x\le \rho} \frac 1x\(\frac{y(x)}x-1\)=\frac{1-\rho}{\rho^2}$. Thus 
$xZ_{d-1}(y(x),\dots,y(x^{d-1})) =x^d F(x) +\Ord{d x^{2d+1}}$. Note further that
$D^{(d)}(x)=\Ord{x^{d+1}}$ since there are no nodes of degree $d$ in trees of size less than
$d+1$. This implies $C^{(d)}(\rho) = C_d\rho^d$ 
with $C_d=C+\Ord{d\rho^d}$ and $C$ as in Theorem~\ref{thm:1}.
\end{proof}

\begin{lem}\label{lem:gamma}
There exist constants $\epsilon^{(d)}, \theta^{(d)}, \tilde{\eta}^{(d)}>0$ and $\theta^{(d)}<\frac{\pi}{2}$ such that \[|\gamma_k^{(d)}(x,u)|=\mathcal{O}(|y(x)|^{k+d})\] uniformly for $x \in \Delta_\varepsilon$ and $u \in \Xi_k$.
\end{lem}

\begin{proof}
 For $l \leq k$ we set
\begin{align*}
 \bar{C}_l^{(d)} = \sup_{\substack{x\in\Delta_\varepsilon\\u \in \Xi_k}}\left|\frac{\gamma_\ell^{(d)}(x,u)}{y(x)^{\ell+d}}\right|.
\end{align*}

First we derive the following inequality, using the recurrence for $y_k^{(d)}$:
\begin{align*}
|y_{\ell+1}^{(d)}(x,u)| &=\left|x\exp\left(\sum_{i\geq1}\frac{1}{i}y_\ell^{(d)}(x^i,u^i)-y(x^i)+y(x^i)\right)\right|\\
&=\left|y(x)\exp\left(\sum_{i\geq1}\frac{1}{i}w_\ell^{(d)}(x^i,u^i)\right)\right|\\
&\leq |y(x)|\exp\left(|w_\ell^{(d)}(x,u)|+\sum_{i\geq2}\frac{1}{i}|w_\ell^{(d)}(x^i,u^i)|\right)\\
&\leq|y(x)|\exp\left(\underbrace{|\gamma_\ell^{(d)}(x,1+\vartheta(u-1))||u-1|}_{\leq\bar{C}_\ell^{(d)}|u-1|} +\sum_{i\geq2}\frac{|u^i-1|}{i}|\gamma_\ell^{(d)}(x^i,1+\vartheta(u^i-1))|\right)
\end{align*}
with $0<\vartheta<1$ and thus $1+\vartheta(u^i-1) \in \Xi_k$. To get an estimate for the second term, we use that $|u^i-1|=|1+u+\cdots+u^{i-1}||u-1|\leq i|u-1|$ as $|u|\leq 1$ and hence $\frac{u^i-1}{i}\leq|u-1|\leq 2$. Further we use $|\gamma_\ell^{(d)}(x^i,1+\vartheta(u^i-1))|\leq |\gamma_\ell^{(d)}(x^i,1)|$, $|y(x)| \leq 1$ and Corollary \ref{cor:sumgamma} to obtain \[|y_{\ell+1}^{(d)}(x,u)|\leq |y(x)|\exp\left(\bar{C}_\ell^{(d)}|u-1|+\mathcal{O}(L^\ell)\right)\]

Using recurrence \eqref{eq:recgamma} leads to:
\begin{align} \label{ineq:Cl}
 \bar{C}_{\ell+1}^{(d)} &=\sup_{\substack{x\in\Delta_\varepsilon\\u \in \Xi_k}}\left|\frac{y_{\ell+1}^{(d)}(x,u)}{y(x)}\right|\left|\frac{\gamma_\ell^{(d)}(x,u)+\sum_{i\geq2}\gamma_\ell^{(d)}(x^i,u^i)u^{i-1}}{y(x)^{\ell+d}}\right|\nonumber\\
&\leq e^{\bar{C}_\ell^{(d)}\frac{\eta}{k}+\mathcal{O}(L^\ell)}(\bar{C}_l^{(d)}+\mathcal{O}(L^l))\nonumber\\
&=\bar{C}_l^{(d)}e^{\bar{C}_\ell^{(d)}\frac{\eta}{k}}(1+\mathcal{O}(L^{\ell})),
\end{align}
where we used Lemma \ref{l:g} to get $|\sum_{i\geq2}\gamma_\ell^{(d)}(x^i,u^i)u^{i-1}|=\mc{O}(\sum_{i\geq2}|y(x^i)|^{\ell+d})$ and hence
\[\sup_{\substack{x\in\Delta_\varepsilon\\u \in \Xi_k}}\left|\frac{\sum_{i\geq2}\gamma_\ell^{(d)}(x^i,u^i)u^{i-1}}{y(x)^{\ell+d}}\right|=\sup\left(\mc{O}\left(\sum_{i\geq2}\left(\frac{y(x^i)}{y(x)}\right)^{\ell+d}\right)\right)=\mc{O}(L^{\ell+d})=\mc{O}(L^{\ell}).\]
We now set
\begin{align*}
 c_0 = \prod\limits_{j \geq 0}(1+\mathcal{O}(L^j)).
\end{align*}

Note that 
\[|\frac{\gamma_0^{(d)}(x,u)}{y(x)^d}| =| \frac{x Z_{d-1}(y(x),y(x^2),\ldots y(x)^{d-1})}{y(x)^d}| = |\mathcal{O}(\frac{x}{y(x)})|=\mathcal{O}(1),\] hence $\bar{C}_0^{(d)}= \sup|\frac{\gamma_0^{(d)}(x,u)}{y(x)^d}|=\mathcal{O}(1)$, too. Thus we can choose $\eta>0$ such that $e^{2 \bar{C}_0^{(d)} c_0 \eta} \leq 2$. For fixed $k$ we get:
\begin{align*}
\bar{C}_l^{(d)} \leq \bar{C}_0^{(d)} \prod\limits_{j < l}(1+\mathcal{O}(L^{j})) e^{2 \bar{C}_0^{(d)} c_0 c \frac{\ell}{k}}\leq 2\bar{C}_0c_0=\mathcal{O}(1).
\end{align*}
The second estimate is clear by the choice of $\eta$ and by $l \leq k$. The first inequality can be obtained from \eqref{ineq:Cl} by induction:
\begin{align*}
 \bar{C}_1^{(d)} &\leq \bar{C}_0^{(d)} (1+\mathcal{O}(L^0)) e^{\bar{C}_0^{(d)}\frac{\eta}{k}} \leq \bar{C}_0^{(d)}\prod\limits_{j < 1}(1+\mathcal{O}(L^{j})) e^{2 \bar{C}_0^{(d)} c_0  \eta\frac{1}{k}}\\
 \bar{C}_{\ell+1}^{(d)} &\leq \bar{C}_\ell^{(d)} e^{\frac{\eta}{k}C_\ell^{(d)}}(1+\mathcal{O}(L^l))\\
&=\prod\limits_{j<\ell}(1+\mathcal{O}(L^j))(1+\mathcal{O}(L^\ell))\bar{C}_0^{(d)} e^{2\bar{C}_0^{(d)}c_0\eta \frac{\ell}{k}}\exp\left(\frac{\eta}{k}\bar{C}_0^{(d)}\Pi_{j<\ell}\underbrace{e^{2\bar{C_0}^{(d)}c_0\eta\frac{\ell}{k}}}_{\leq 2^{\frac{\ell}{k}}\leq 2}\right)\\
&\leq \bar{C}_0 \prod\limits_{j<l+1}(1+\mathcal{O}(L^{j})) e^{2 \bar{C}_0 c_0 \eta \frac{l+1}{k}})
\end{align*}
\end{proof}

For the second derivatives with respect to $u$ of $y_k^{(d)}(x,u)$, $\gamma_k^{(d)[2]}(x,u)$, we find that
\begin{lem}\label{l:g2}
 Suppose that $|x| \leq \rho-\eta$ for some $\eta >0$ and $|u| \leq 1$. Then
\begin{align}
 \gamma_k^{(d)[2]}(x,u) = \mathcal{O}(y(|x|)^{k+d})
\end{align}
uniformly. There also exist constants $\epsilon, \theta, \tilde{\eta}$ such that uniformly for $u \in \Xi_k$ and $x\in\Delta_\epsilon$
\begin{align}
 \gamma_k^{(d)[2]}(x,u) = \mathcal{O}(k y(|x|)^{k+d})
\end{align}
\end{lem}

\begin{proof}
 Derivation of \eqref{eq:recgamma} leads to the recurrence
\begin{align*}
 \gamma_{k+1}^{(d)[2]}(x,u)&= y_{k+1}^{(d)}(x,u) \left(\sum\limits_{i\geq1} \gamma_k^{(d)}(x^i,u^i)u^{i-1}\right)^2 \\
&+y_{k+1}^{(d)}(x,u)\sum\limits_{i\geq1}i \gamma_k^{(d)[2]}(x^i,u^i)u^{2(i-1)} \\
&+y_{k+1}^{(d)}(x,u)\sum\limits_{i\geq2}(i-1) \gamma_k^{(d)}(x^i,u^i)u^{i-2}
\end{align*}
with initial condition $\gamma_0^{(d)[2]}(x,u)=0$.

For $|x|<\rho-\eta$, for some $\eta>0$ and for $|u|\leq1$ we have $|\gamma_{k}^{(d)[2]}(x,u)| \leq  \gamma_{k}^{(d)[2]}(|x|,1)$. Thus, in this case we can restrict ourselves to non-negative real $x \leq \rho - \eta$. 

By using the bounds $\gamma_{k}^{(d)}(x,1) \leq C^{(d)}_k y(x)^{k+d}$ from Lemma \ref{lem:gamma}, $\sum_{i \geq 2}\gamma_k^{(d)}(x^i,u^i) = \mathcal{O}(L^k)$ from Corollary \ref{cor:sumgamma} and the induction hypothesis $\gamma_{k}^{(d)[2]}(x,1) \leq D_k^{(d)} y(x)^{k+d}$, we can derive the following upper bound from the above:

\begin{align*}
 \gamma_{k+1}^{(d)[2]}(x) &= y(x) \left(\sum\limits_{i \geq 1}  \gamma_{k}^{(d)}(x^i)\right)^2+y(x)\sum\limits_{i \geq1} i \gamma_{k}^{(d)[2]}(x^i) +y(x) \sum\limits_{i \geq2}(i-1) \gamma_{k}^{(d)}(x^i)\\
&\leq y(x)\left[\left(C_k^{(d)}y(x)^{k+d}+\mathcal{O}(L^k)\right)^2 + D_k^{(d)}(\sum_{i\geq1}iy(x^i)^{k+d})+\sum_{i\geq2}C_k^{(d)}y(x^i)^{k+d})\right]\\
&\leq y(x)^{k+d+1} \left[C_k^{(d)2}y(x)^{k+d} + C_k^{(d)}\mc{O}(L^k)\right.\\
&\hspace{3cm} \left.+ D_k\left(1+\sum\limits_{i \geq2} i \frac{y(x^i)^{k+d}}{y(x)^{k+d}}\right)+ C \sum\limits_{i \geq 2}(i-1) \frac{y(x^i)^{k+d}}{y(x)^{k+d}}\right]\\
& \leq y(x)^{k+d+1} (C^{(d)2}_k y(\rho-\eta)^{k+d} +D_k(1+ \mathcal{O}(L^k)) + \mathcal{O}(L^k)).
\end{align*}

Consequently

\begin{align*}
 D_{k+1}^{(d)} = (C_k^{(d)2}y(\rho-\eta)^{k+d} + D_k(1+ \mathcal{O}(L^k)) + \mathcal{O}(L^k))
\end{align*}

which leads to $D_k^{(d)} = \mathcal{O}(1)$ as $k \to  \infty$.\\

To prove the second property we use the same constants $\epsilon, \theta, \tilde{\eta}$ as in Lemma \ref{lem:gamma} and set: 
\begin{align*}
 \bar{D}_\ell^{(d)} = \sup_{\substack{x\in\Delta_\varepsilon\\u \in \Xi_k}}|\frac{\gamma_\ell^{(d)[2]}(x,u)}{y(x)^{\ell+d}}|,
\end{align*}

We use the already known bound $|\gamma_\ell^{(d)}(x,u)| \leq \bar{C}^{(d)}|y(x)^{k+d}|$ and by similar considerations as in the proof of Lemma \ref{lem:gamma} we get:

\begin{align*}
\bar{D}_{\ell+1}^{(d)} &= \sup_{\substack{x\in\Delta_\varepsilon\\u \in \Xi_k}}\left|\frac{y_{\ell+1}^{(d)}(x,u)}{y(x)}\right|\\
&\quad \times \left|\frac{\left(\sum\limits_{i\geq1} \gamma_\ell^{(d)}(x^i,u^i)u^{i-1}\right)^2 + \sum\limits_{i\geq1}i \gamma_\ell^{(d)[2]}(x^i,u^i)u^{2(i-1)} + \sum\limits_{i\geq2}(i-1) \gamma_\ell^{(d)}(x^i,u^i)u^{i-2}}{y(x)^{\ell+d}}\right|\\
&\leq \bar{D}_\ell^{(d)} e^{\bar{C}^{(d)}\frac{\eta}{k}}(1+ \mathcal{O}(L^\ell)) + C^{(d)2} e^{\bar{C}^{(d)}\frac{\eta}{k}} + \mathcal{O}(L^\ell))\\
&\leq \alpha_\ell^{(d)}\bar{D}_\ell^{(d)} + \beta_\ell^{(d)}
\end{align*}

with $\alpha_\ell^{(d)} = e^{\bar{C}^{(d)}\frac{\eta}{k}}(1+ \mathcal{O}(L^\ell))$ and $ \beta_\ell^{(d)}=C^{(d)2} e^{\bar{C}^{(d)}\frac{\eta}{k}} + \mathcal{O}(L^\ell))$.  Thus

\begin{align*}
 \bar{D}_{k}^{(d)} &\leq \alpha_{k-1}^{(d)}(\alpha_{k-2}^{(d)}(\ldots (\alpha_0^{(d)}D_0 + \beta_0^{(d)}) \ldots)\beta_{k-2}^{(d)})+ \beta_{k-1}^{(d)}\\
&= \sum\limits_{j=0}^{k-1}\beta_j^{(d)} \prod\limits_{i=j+1}^{k-1}\alpha_i^{(d)}+\alpha_0^{(d)}\bar{D}_0^{(d)}\\
&\leq k \max\limits_j \beta_j^{(d)}e^{\bar{C}^{(d)}c}\prod\limits_{i \geq 0} (1+\mathcal{O}(L^i))\\
&=\mathcal{O}(k),
\end{align*}

which completes the proof of the Lemma.
\end{proof}

\begin{lem}\label{l:wkc}
Let $\epsilon, \theta, \tilde{\eta}$ and $C_k^{(d)}(x)$ be as in Lemma \ref{l:g} and Lemma \ref{lem:gamma}. Then 
\begin{align}\label{eq:wkgk}
 w_k^{(d)}(x,u) = C_k^{(d)}(x)(u-1)y(x)^{k+d}(1+\mathcal{O}(k|u-1|))
\end{align}

uniformly for $x \in \Delta_\epsilon$ and $u\in \Xi_k$.
Furthermore we have for $|x| \leq \rho + \eta$ and $|u| \leq 1$
\begin{align}\label{eq:skgk}
 \Sigma_k^{(d)}(x,u) = \tilde{C}_k^{(d)}(x) (u-1)y(x^2)^{k+d} + \mathcal{O}(|u-1|^2y(|x|^2)^{k+d}),
\end{align}
 
where the analytic functions $ \tilde{C}_k^{(d)}(x)$ are given by 

\begin{align*}
  \tilde{C}_k^{(d)}(x) = \sum\limits_{i \geq 2}  C_k^{(d)}(x^i)\left(\frac{y(x^i)}{y(x^2)}\right)^{k+d}
\end{align*}

and have a uniform limit $ \tilde{C}^{(d)}(x)$ with convergence rate
\begin{align*}
  \tilde{C}_k^{(d)}(x)= \tilde{C}^{(d)}(x) + \mathcal{O}(L^k)
\end{align*}

for some constant $L$ with $0<L<1$. 
\end{lem}

\begin{proof}
 To prove the first statement, we expand $w_k^{(d)}(x,u)$ into a Taylor polynomial of degree $2$ around $u=1$ and apply Lemmas \ref{l:g} and \ref{l:g2}. 

To prove the second statement, we again use Taylor series. Note that for $i \geq2$ we have $|x^i| < \rho-\eta$ if $|x| < \rho+\eta$ and $\eta$ is sufficiently small. We get 
\begin{align*}
 w_k^{(d)}(x^i,u^i) = C_k^{(d)}(x^i)(u^i-1)y(x^i)^{k+d} + \mathcal{O}(|u^i-1|^2y(|x^i|)^{k+d})
\end{align*}

and consequently 
\begin{align*}
 \Sigma_k^{(d)}(x,u) &= \sum\limits_{i \geq 2} \frac{1}{i} C_k^{(d)}(x^i)(u^i-1)y(x^i)^{k+d}+ \mathcal{O}(|u-1|^2 y(|x^2|)^{k+d})\\
&=(u-1)\tilde{C}_k^{(d)}(x)y(x^2)^{k+d} + \mathcal{O}(|u^i-1|^2y(|x^i|)^{k+d}),
\end{align*}

where we used the property that 
\begin{align*}
 \sum\limits_{i \geq 2} C_k^{(d)}(x^i) \frac{u^i-1}{i(u-1)}\frac{y(x^i)^{k+d}}{y(x^2)^{k+d}} 
&=\sum\limits_{i \geq 2} C_k^{(d)}(x^i) \frac{(1+u+\cdots+u^{i-1})}{i}\frac{y(x^i)^{k+d}}{y(x^2)^{k+d}}\\
&=\tilde{C}_k^{(d)}(x)+\mathcal{O}(\tilde{C}_k^{(d)}(x)(u-1))
\end{align*}

represents an analytic function in $x$ and $u$, and thus its leading term, as $u \to\infty$, is our function $\tilde{C}_k^{(d)}(x)$. Finally, since $C_k^{(d)}(x) = C^{(d)}(x) + \mathcal{O}(L^k)$ it follows that $\tilde{C}_k^{(d)}(x)$ has a limit $\tilde{C}^{(d)}(x)$ with the same order of convergence.
\end{proof}

\begin{lem}\label{lem:wkd}
 For $x \in \Delta_\epsilon$ and $u \in \Xi_k$ (with the constants $\epsilon, \theta, \tilde{\eta}$ as in Lemma \ref{lem:gamma}) we have
\begin{align*}
 w_k^{(d)}(x,u) = \frac{(u-1)y(x)^{k+d} C_k^{(d)}(x)}{1-\frac{y(x)^d C_k^{(d)}(x)\cdot(u-1)}{2}\frac{1-y(x)^k}{1-y(x)} + \mathcal{O}(|u-1|)}
\end{align*}
\end{lem}

\begin{proof}
 $w_k^{(d)}(x,u)$ satisfy the recursive relation
\begin{align*}
w_{k+1}^{(d)}(x,u)&=x \exp \left(\sum \limits_{i \geq 1} \frac{1}{i}y_k^{(d)}(x^i,u^i)\right) - y(x) \\
&=x \exp \left(\sum\limits_{i \geq 1} \frac{1}{i} \left(w_k^{(d)}(x^i,u^i) + y(x^i)\right)\right) -y(x)\\
&=y(x) \left( \exp \left(w_k^{(d)}(x,u) + \Sigma_k^{(d)}(x,u)\right)-1\right),
\end{align*}
and further, since by Lemma \ref{l:wkc} it follows that $\Sigma_k^{(d)}(x,u) = \mathcal{O}(w_k^{(d)}(x,u)L^k)= \mathcal{O}(w_k^{(d)}(x,u))$ (for brevity, we omit the variables now),
\begin{align*}
&w_{k+1}^{(d)}=y\left[(w_k^{(d)}+\Sigma_k^{(d)})+\frac{(w_k^{(d)}+\Sigma_k^{(d)})^2}{2}+\mc{O}\left((w_k^{(d)}+\Sigma_k^{(d)})^3\right)\right]\\
&=y(w_k^{(d)}+\Sigma_k^{(d)})\left(1+\frac{(w_k^{(d)}+\Sigma_k^{(d)})}{2}+\mc{O}\left((w_k^{(d)}+\Sigma_k^{(d)})^2\right)\right)\\
&=yw_k^{(d)}\left(1+\frac{\Sigma_k^{(d)}}{w_k^{(d)}}\right)\left(1+\frac{w_k^{(d)}}{2}+\mc{O}(\Sigma_k^{(d)})+\mc{O}\left((w_k^{(d)})^2\right)\right).
\end{align*}
From there, we obtain
\begin{align*}
 \frac{y}{w_{k+1}^{(d)}}\cdot\left(1+\frac{\Sigma_k^{(d)}}{w_k^{(d)}}\right)  &= \frac{1}{w_k^{(d)}}\frac{1}{\left(1+\frac{w_k^{(d)}}{2}+\mc{O}(\Sigma_k^{(d)})+\mc{O}\left((w_k^{(d)})^2\right)\right)}\\
&=\frac{1}{w_k^{(d)}} \left(1-\frac{w_k^{(d)}}{2} + \mathcal{O}(\Sigma_k^{(d)}) + \mathcal{O}(w_k^{(d)2})\right)\\
&=\frac{1}{w_k^{(d)}}-\frac{1}{2} +  \mathcal{O}\left(\frac{\Sigma_k^{(d)}}{w_k^{(d)}}\right)+\mathcal{O}(w_k^{(d)}).
\end{align*}
This leads us to a recursion
\begin{align*}
\frac{y^{k+1}}{w_{k+1}^{(d)}}=\frac{y^{k}}{w_k^{(d)}}-\frac{\Sigma_k^{(d)}\cdot y(x)^{k+1}}{w_k^{(d)}w_{k+1}^{(d)}}-\frac{1}{2} y(x)^k+ \mathcal{O}\left(\frac{\Sigma_k^{(d)}\cdot y^k}{w_k^{(d)}}\right)+ \mathcal{O}(w_k^{(d)}y^k) 
\end{align*}
which we can solve to 
\begin{align}\label{eq:10a}
 \frac{y^{k}}{w_{k}^{(d)}}&= \frac{1}{w_0^{(d)}} - \sum\limits_{\ell=0}^{k-1} \frac{\Sigma_\ell^{(d)}\cdot y(x)^{\ell+1}}{w_\ell^{(d)}w_{\ell+1}^{(d)}} - \frac{1}{2}\sum_{\ell=0}^{k-1}y^\ell+ \mathcal{O}\left(\sum\limits_{\ell=0}^{k-1}\frac{\Sigma_\ell^{(d)}\cdot y^\ell}{w_\ell^{(d)}}\right) +  \mathcal{O}(\sum\limits_{\ell=0}^{k-1}w_\ell^{(d)} y^\ell) \nonumber\\
&=\frac{1}{w_0^{(d)}}\left(1 - w_0^{(d)}\sum\limits_{\ell=0}^{k-1} \frac{\Sigma_\ell^{(d)}\cdot y^{\ell+1}}{w_\ell^{(d)}w_{\ell+1}^{(d)}} - w_0^{(d)}\frac{1}{2}\frac{1-y^k}{1-y}+ \mathcal{O}(w_0^{(d)}\frac{1-L^k}{1-L}) +  \mathcal{O}((w_0^{(d)})^2\frac{1-y^{2k}}{1-y^2})\right),
\end{align}
where we used that $\frac{\Sigma_\ell^{(d)}y^\ell}{w_\ell^{(d)}}=\mc{O}(L^\ell)$ and that by Lemma \ref{l:wkc} $w_k^{(d)}=\mc{O}(yw_{k-1}^{(d)} )= \mc{O}(y^kw_0^{(d)})$. Again we apply Lemma \ref{l:wkc} and \eqref{eq:rec} to obtain
\begin{align*}
&w_0^{(d)}\sum\limits_{\ell=0}^{k-1} \frac{\Sigma_\ell^{(d)}\cdot y(x)^{\ell+1}}{w_\ell^{(d)}w_{\ell+1}^{(d)}}=(u-1)xZ_{d-1} \sum\limits_{\ell=0}^{k-1} \frac{\tilde{C}_\ell^{(d)}(u-1)y(x^2)^{\ell+d} + \mathcal{O}(|u-1|^2y(|x|^2)^{\ell+d})}{C_\ell^{(d)}C_{\ell+1}^{(d)}y^{2(\ell+d)+1}(u-1)^2(1+\mathcal{O}(\ell|u-1|))}y^{\ell+1}\\
&\qquad=\frac{xZ_{d-1}}{y(x)^d} \sum\limits_{\ell=0}^{k-1} \frac{\tilde{C}_\ell^{(d)}y(x^2)^{\ell+d} + \mathcal{O}(|u-1|^2y(|x|^2)^{\ell+d})}{C_\ell^{(d)}C_{\ell+1}^{(d)}y(x)^{\ell+d}(1+\mathcal{O}(\ell|u-1|))}\\
&\qquad=\frac{xZ_{d-1}}{y(x)^d} \left[\sum\limits_{\ell=0}^{k-1}\frac{\tilde{C}_\ell^{(d)}}{C_\ell^{(d)}C_{\ell+1}^{(d)}}\frac{y(x^2)^{\ell+d}}{y(x)^{\ell+d}} + \sum\limits_{\ell=0}^{k-1}\frac{ \mathcal{O}(|u-1|^2y(|x|^2)^{\ell+d})}{C_\ell^{(d)}C_{\ell+1}^{(d)}y(x)^{\ell+d}}\right]\left(\frac{1}{1+\mc{O}(\ell|u-1|)}\right)\\
&\qquad=\frac{xZ_{d-1}}{y(x)^d} \left[\sum\limits_{\ell=0}^{k-1}\frac{\tilde{C}_\ell^{(d)}}{C_\ell^{(d)}C_{\ell+1}^{(d)}}\frac{y(x^2)^{\ell+d}}{y(x)^{\ell+d}} + \sum\limits_{\ell=0}^{k-1}\underbrace{\frac{ \mathcal{O}(|u-1|^2y(|x|^2)^{\ell+d})}{C_\ell^{(d)}C_{\ell+1}^{(d)}y(x)^{\ell+d}}}_{=\mc{O}(|u-1|^2L^\ell)}\right](1+\mc{O}(\ell|u-1|))\\
&\qquad \frac{xZ_{d-1}}{y(x)^d} \left[\sum\limits_{\ell=0}^{k-1}\frac{\tilde{C}_\ell^{(d)}}{C_\ell^{(d)}C_{\ell+1}^{(d)}}\frac{y(x^2)^{\ell+d}}{y(x)^{\ell+d}}+ \sum\limits_{\ell=0}^{k-1}\frac{\tilde{C}_\ell^{(d)}}{C_\ell^{(d)}C_{\ell+1}^{(d)}}\underbrace{\frac{y(x^2)^{\ell+d}}{y(x)^{\ell+d}}}_{=\mc{O}(L^\ell)}\mc{O}(\ell|u-1|)+\mc{O}(|u-1|^2)\right]\\
&\qquad=c_k^{(d)} + \mathcal{O}(|u-1|),
\end{align*}
where $c_k^{(d)}$ denotes the first sum. Note that $\frac{xZ_{d-1}}{y(x)^d}=\mc{O}(1)$. \\
Now turn back to \eqref{eq:10a} and observe that $w_0^{(d)}\frac{1-y^{2k}}{1-y^2} = \mathcal{O}(k|u-1|y(x)^d) =\mathcal{O}(y(x)^d)=\mathcal{O}(1)$ if $k|u-1| \leq \tilde{\eta}$. Thus, we obtain the following representation for $w_k^{(d)}(x,u)$:
\begin{align*}
 w_k^{(d)} = \frac{w_0^{(d)} y^{k}}{1-c_k^{(d)}(x) - \frac{w_0^{(d)}}{2}\frac{1-y^k}{1-y} + \mathcal{O}(|u-1|)}.
\end{align*}
We use the expressions 
\[C_{k+1}^{(d)}=\sum_{i\geq 1}C_k^{(d)}\frac{y(x^i)^{k+d}}{y(x)^{k+d}}\! \textrm{ and }\!
\tilde{C}_k^{(d)} = \sum_{i\geq2} C_k^{(d)}\frac{y(x^i)^{k+d}}{y(x^2)^{k+d}},\] which are
consequences of Lemmas \ref{l:g}, Equation \eqref{eq:recgamma} and Lemma~\ref{l:wkc}, to obtain
\begin{align}\label{eq:tildeC}
 \tilde{C}_k^{(d)}(x)=(C_{k+1}^{(d)}(x)-C_k^{(d)}(x))\left(\frac{y(x)}{y(x^2)}\right)^{k+d}.
\end{align}
This provides the telescope sum:
\begin{align}
 c_k^{(d)} &= \frac{x Z_{d-1}}{y(x)^d} \sum\limits_{\ell=0}^{k-1} \frac{C_{\ell+1}^{(d)} - C_\ell^{(d)}}{C_\ell^{(d)}C_{\ell+1}^{(d)}}\\
 &= \frac{x Z_{d-1}}{y(x)^d}\left(\frac{1}{C_0^{(d)}} - \frac{1}{C_k^{(d)}}\right)\\
\end{align}

and hence, since $C_0^{(d)} = \frac{\gamma_0^{(d)}}{y(x)^d}=\frac{xZ_{d-1}}{y(x)^d}$, we get
\[
1-c_k^{(d)}(x) = \frac{x Z_{d-1}(y(x),\dots,y(x^{d-1}))}{y(x)^d C_k^{(d)}(x)},
\]
which yields the result. 
\end{proof}

It is now easy to proof Theorem \ref{th:exp}. With $x=\rho(1+\frac{s}{n})$, $u=e^{\frac{it}{\sqrt{n}}}$, $d$ and $t\neq0$ fixed, $k=\kappa\sqrt{n}$ and \eqref{exp} we obtain the expansions:

\begin{align*}
u-1 &\sim \frac{it}{\sqrt{n}}\\
1-y(x) &\sim b \sqrt{\frac{-\rho s}{n}} \\
y(x)^k&\sim 1-kb \sqrt{\frac{-\rho s}{n}} + \cdots \sim e^{-\kappa b \sqrt{-\rho s}}\\
y(x)^d&\sim 1-db \sqrt{\frac{-\rho s}{n}} + \cdots \sim 1.
\end{align*}

Since the functions $C_k^{(d)}(x)$ are continuous and uniformly convergent to $C^{(d)}(x)$, they
are also uniformly continuous and thus $C_k^{(d)}(x)\sim C^{(d)} (\rho) = C_d \rho^d$. This leads to

\begin{align}
w_k^{(d)}(x,u) &\sim \frac{\frac{it}{\sqrt{n}}C_d\rho^{d} e^{-\kappa b \sqrt{-\rho
s}}}{1-\frac{it}{\sqrt{n}}C_d\rho^{d}\left(\frac{1}{2}\frac{1-e^{-\kappa b \sqrt{-\rho s}}}{b\sqrt{\frac{-\rho s}{n}}}\right)}\nonumber\\
&=\frac{1}{\sqrt{n}} \cdot \frac{\sqrt{-s}itC_d\rho^{d}e^{-\kappa b \sqrt{-\rho s}}}{\sqrt{-s} -
\frac{itC_d\rho^{d}}{2b\sqrt{\rho}}\left(1-e^{-\kappa b \sqrt{-\rho s}}\right)}\nonumber\\
&=\frac{C_d \rho^d}{\sqrt{n}} \cdot \frac{it\sqrt{-s}e^{\frac{1}{2}-\kappa b \sqrt{-\rho
s}}}{\sqrt{-s}e^{\frac{1}{2}\kappa b \sqrt{-\rho s}} - \frac{itC_d\rho^{d.}}{b\sqrt{\rho}}\sinh(\frac{1}{2}\kappa b \sqrt{-\rho s})}
\end{align}

\section{Finite dimensional limiting distributions} \label{sec:5}

First we consider the case $m=2$. The computation of the $2$-dimensional limiting distribution shows the general method of the proof. Iterative applications of the arguments will eventually prove Theorem \ref{th:3}. 

\begin{thm}\label{th:expmult}
Let $x = \rho (1+\frac{s}{n})$, $u_1=e^{\frac{it_1}{\sqrt{n}}}, u_2=e^{\frac{it_2}{\sqrt{n}}}$, $k=\kappa \sqrt{n}$ and $h=\eta \sqrt{n}$. Moreover, assume that $|\arg s| \geq \Theta > 0$ and, as $n \to \infty$, we have $s=\mathcal{O}(\log^2 n)$, whereas $\kappa$, $t_1$ and $t_2$ are fixed. Then, for large $d$, $w_{k,h}^{(d)}(x,u)$ admits the local representation 

\begin{align}\label{eq:wkdmult}
&w_{k,h}^{(d)}(x,u,v) \sim \frac{C_d \rho^d}{\sqrt{n}}\nonumber\\
 &\times \frac{\left( it_2 +\frac{it_1 \sqrt{-s} e^{(-\frac{1}{2}\kappa b \sqrt{-\rho s
})}}{\sqrt{-s}e^{(\frac{1}{2}\kappa b \sqrt{-\rho s })}-\frac{it_1 C_d
\rho^d}{\sqrt{\rho}b}\sinh{(\frac{1}{2}\kappa b \sqrt{-\rho s})})}\right)\sqrt{-s}
e^{(-\frac{1}{2}\xi b \sqrt{-\rho s })}}{\sqrt{-s}e^{(\frac{1}{2}\xi b \sqrt{-\rho s
})}-\frac{C_d\rho^d}{b\sqrt{\rho}}\left(it_2+\frac{it_1 \sqrt{-s} e^{(-\frac{1}{2}\kappa b
\sqrt{-\rho s })}}{\sqrt{-s}e^{(\frac{1}{2}\kappa b \sqrt{-\rho s })}-\frac{it_1 C_d \rho^d}{\sqrt{\rho}b}\sinh{(\frac{1}{2}\kappa b \sqrt{-\rho s})})}\right)\sinh{(\frac{1}{2}\xi b \sqrt{-\rho s})})}.
\end{align}
\end{thm}

\begin{proof}
Note that $y_{k,h}^{(d)}(x,u_1,1) = y_k^{(d)}(x,u_1)$ and $y_{k,h}^{(d)}(x,1,u_2) =
y_{k+h}(x,u_2)$. Considering the first derivative, we denote by 
\begin{align*}
\frac{\partial}{\partial u_1} y_{k,h}^{(d)}(x,u_1,u_2)&=:\gamma_{k,h}^{(d)[u_1]}(x,u_1,u_2)\\
\frac{\partial}{\partial u_2} y_{k,h}^{(d)}(x,u_1,u_2)&=:\gamma_{k,h}^{(d)[u_2]}(x,u_1,u_2),
\end{align*}
and by simple induction, we observe that:
\begin{align}\label{eq:mult}
\gamma_{k,h}^{(d)[u_1]}(x,1,1) &= \gamma_k^{(d)}(x,1) = \gamma_k^{(d)}(x)=C_k^{(d)}(x)y(x)^{k+d}\nonumber\\
\gamma_{k,h}^{(d)[u_2]}(x,1,1) &= \gamma_{k+h}^{(d)}(x,1) = \gamma_{k+h}^{(d)}(x)=C_{k+h}^{(d)}(x)y(x)^{k+h+d}.
\end{align}
As $|\gamma_{k,h}^{(d)[u_i]}(x,u_1,u_2)| \leq \gamma_{k,h}^{(d)[u_i]}(x,1,1)$ for $i=1,2; u_1\in \Xi_k,u_2\in \Xi_{k+h}$, and $|x|\leq \rho$ it follows that
$|\gamma_{k,h}^{(d)[u_1]}(x,u_1,u_2)| = \mc{O}(y(x)^{k+d})$ and $\gamma_{k,h}^{(d)[u_2]}(x,u_1,u_2)\leq \mc{O}(y(x)^{k+h+d})$ in the same regions. To be more precise, we can prove the following analogue to Lemma \ref{lem:gamma}.

\begin{lem}\label{lem:gamma2}
There exist constants $\epsilon, \vartheta, \tilde{\eta}_1,\tilde{\eta}_2$, such that or $x\in\Delta_\epsilon, u_1\in \Xi_{k}$ and $u_2\in\Xi_{k+h}$ 
\begin{align*}
 \gamma_{k,h}^{(d)[u_1]}(x,u_1,u_2)+\gamma_{k,h}^{(d)[u_2]}(x,u_1,u_2)=\mc{O}(|y(x)|^{k+d})
\end{align*}
\end{lem}
\begin{proof}
 Set
\begin{align*}
 C_{\ell,h}^{(d)[u_1]}&=\sup_{\substack{x\in\Delta_\epsilon\\u_1\in\Xi_k,u_2\in\Xi_{k+h}}}\left|\frac{\gamma_{\ell,h}^{(d)[u_1]}(x,u_1,u_2)}{y(x)^{k+d}}\right|,\\
C_{\ell,h}^{(d)[u_2]}&=\sup_{\substack{x\in\Delta_\epsilon\\u_1\in\Xi_k,u_2\in\Xi_{k+h}}}\left|\frac{\gamma_{\ell,h}^{(d)[u_2]}(x,u_1,u_2)}{y(x)^{k+h+d}}\right|.
\end{align*}
As in the proof of Lemma \ref{l:g2}, we apply Taylor's theorem (in two variables) to get
\begin{align*}
 |y_{\ell+1,h}^{(d)}(x,u_1,u_2)|&=|y(x)|\exp&&\!\!\!\!\!\!\left(|w_\ell^{(d)}(x,u_1,u_2)|+\sum_{i \geq 2}\frac{|w_\ell^{(d)}(x^i,u_1^i,u_2^i)|}{i}\right)\\
 &\leq|y(x)|\exp&&\!\!\!\!\!\!\left(\gamma_{\ell,h}^{(d)[u_1]}(x,1+\vartheta_1(u_1-1),1+\vartheta_2(u_2-1))(u_1-1)\right.\\
&&&\!\!\!\!\!\!+\left.\gamma_{\ell,h}^{(d)[u_2]}(x,1+\vartheta_1(u_1-1),1+\vartheta_2(u_2-1))(u_2-1)\right.\\
&&&\!\!\!\!\!\!+\left.\sum_{i\geq2}\gamma_{\ell,h}^{(d)[u_1]}(x^i,1+\vartheta_1(u_1^i-1),1+\vartheta_2(u_2^i-1))\frac{(u_1^i-1)}{i}\right.\\
&&&\!\!\!\!\!\!+\left.\gamma_{\ell,h}^{(d)[u_2]}(x^i,1+\vartheta_1(u_1^i-1),1+\vartheta_2(u_2^i-1))\frac{u_2^i-1}{i}\right)
\end{align*}
\[ \left|y_{\ell+1,h}^{(d)}(x,u_1,u_2)\right|\leq|y(x)|\exp\left(C_{\ell,h}^{(d)[u_1]}|u_1-1|y(x)^{\ell+d}+C_{\ell,h}^{(d)[u_2]}|u_2-1|y(x)^{\ell+d}+\mc{O}(L^\ell)\right),\]
where we use that, for $i\geq2$,
\begin{align*} 
|\gamma_{\ell,h}^{(d)[u_1]}(x^i,1+\vartheta_1(u_1^i-1),1+\vartheta_2(u_2^i-1))|&\leq|\gamma_{\ell,h}^{(d)[u_1]}(x^i,1,1)| \;\;\textrm{and}\\ 
|\gamma_{\ell,h}^{(d)[u_2]}(x^i,1+\vartheta_1(u_1^i-1),1+\vartheta_2(u_2^i-1))|&\leq|\gamma_{\ell,h}^{(d)[u_2]}(x^i,1,1)|.
\end{align*}
By using recursion \eqref{eq:rec2dim} and Lemma \ref{lem:gamma}, we obtain
\begin{align*}
 C_{\ell+1,h}^{(d)[u_1]}&=\sup_{\substack{x\in\Delta_\epsilon\\u_1\in\Xi_k,u_2\in\Xi_{k+h}}}\left|\frac{y_{\ell+1,h}^{(d)}(x,u_1,u_2)}{y(x)}\right|\left|\frac{\gamma_{\ell,h}^{(d)[u_1]}(x,u_1,u_2)+\sum_{i\geq2}\gamma_{\ell,h}^{(d)[u_1]}(x^i,u_1^i,u_2^i)}{y(x)^{k+d}}\right|\\
&\leq  \exp\left(C_{\ell,h}^{(d)[u_1]}\frac{\eta_1}{k}+C_{\ell,h}^{(d)[u_2]}\frac{\eta_2}{k}+\mc{O}(L^\ell)\right)\left(C_{\ell,h}^{(d)[u_1]}+\mc{O}(L^\ell)\right)\\
&=C_{\ell,h}^{(d)[u_1]}\exp\left( C_{\ell,h}^{(d)[u_1]}\frac{\eta_1}{k}+C_{\ell,h}^{(d)[u_2]}\frac{\eta_2}{k}\right)\left(1+\mc{O}(L^\ell)\right),
\end{align*}
and analogously
\[C_{\ell+1,h}^{(d)[u_2]}\leq C_{\ell,h}^{(d)[u_2]}\exp\left(C_{\ell,h}^{(d)[u_1]}\frac{\eta_1}{k}+C_{\ell,h}^{(d)[u_2]}\frac{\eta_2}{k}\right)\left(1+\mc{O}(L^\ell)\right).\]
We choose $\eta_1$ and $\eta_2$ such that $e^{2c_o(C_{0,h}^{(d)[u_1]}\eta_1+C_{0,h}^{(d)[u_2]}\eta_2}\leq 2$. Then, by induction we get
\begin{align*}
C_{\ell,h}^{(d)[u_1]}&\leq C_{0,h}^{(d)[u_1]}\prod_{j<\ell}(1+\mc{O}(L^j))e^{2c_o(C_{0,h}^{(d)[u_1]}\eta_1+C_{0,h}^{(d)[u_2]}\eta_2)\frac{\ell}{k}}\leq2C_{0,h}^{(d)[u_1]}c_0=\mc{O}(1),\\
C_{\ell,h}^{(d)[u_2]}&\leq C_{0,h}^{(d)[u_2]}\prod_{j<\ell}(1+\mc{O}(L^j))e^{2c_o(C_{0,h}^{(d)[u_1]}\eta_1+C_{0,h}^{(d)[u_2]}\eta_2)\frac{\ell}{k}}\leq2C_{0,h}^{(d)[u_2]}c_0=\mc{O}(1).
\end{align*}
Note therefore that
\begin{align*}
 C_{0,h}^{(d)[u_1]}&=\sup_{\substack{x\in\Delta_\epsilon\\u_1\in\Xi_k,u_2\in\Xi_{k+h}}}\left|\frac{xZ_{d-1}(y_h^{(d)}(x,u_2),\ldots,y_h^{(d)}(x^{d-1},u_2^{d-1}))}{y(x)^d}\right|=\mc{O}(1)\\
 C_{0,h}^{(d)[u_1]}&=\sup_{\substack{x\in\Delta_\epsilon\\u_1\in\Xi_k,u_2\in\Xi_{k+h}}}\left|\frac{\gamma_{0,h}^{(d)[u_2]}(x,u_1,u_2)+(u_1-1)x\frac{\partial}{\partial u_2}Z_{d-1}(y_h(x,u_2),\ldots,y_h(x^{d-1},u_2^{d-1}))}{y(x)^{h+d}}\right|\\
&=\mc{O}\left(\sup_{\substack{x\in\Delta_\epsilon\\u_1\in\Xi_k,u_2\in\Xi_{k+h}}}\left|\frac{\gamma_{0,h}^{(d)[u_2]}(x,u_1,u_2)}{y(x)^{h+d}}\right|\right)=\mc{O}(1)
\end{align*}

\end{proof}

Let \begin{align*}
\gamma_{k,h}^{(d)[2u_1]}(x,u_1,u_2)&:=\frac{\partial^2}{\partial u_1^2}y_{k,h}^{(d)}(x,u_1,u_2)\\
\gamma_{k,h}^{(d)[2u_2]}(x,u_1,u_2)&:=\frac{\partial^2}{\partial u_2^2}y_{k,h}^{(d)}(x,u_1,u_2)\\
\gamma_{k,h}^{(d)[u_1u_2]}(x,u_1,u_2)&:=\frac{\partial^2}{\partial u_1\partial u_2}y_{k,h}^{(d)}(x,u_1,u_2)\\
\gamma_{k,h}^{(d)[2]}(x,u_1,u_2)&:=\gamma_{k,h}^{(d)[2u_1]}(x,u_1,u_2)+\gamma_{k,h}^{(d)[2u_2]}(x,u_1,u_2)+\gamma_{k,h}^{(d)[u_1u_2]}(x,u_1u_2)
\end{align*}

\begin{lem}\label{l:g22}    
 For $|u_1| \leq 1, |u_2| \leq 1$ and for $|x| \leq \rho - \eta$ for some $\eta > 0$
\begin{align*}
 \gamma_{k,h}^{(d)[2u_1]}(x,u_1,u_2) &= \mathcal{O}(y(|x|)^{k+d})\\
 \gamma_{k,h}^{(d)[u_1u_2]}(x,u_1,u_2) &= \mathcal{O}(y(|x|)^{k+h+2d-1})\\
 \gamma_{k,h}^{(d)[2u_2]}(x,u_1,u_2) &= \mathcal{O}(y(|x|)^{k+h+d})
\end{align*}
uniformly. Furthermore, for $x \in \Delta_\epsilon, u_1 \in \Xi_k$ and $u_2 \in \Xi_{k+h}$
\begin{align*}
  \gamma_{k,h}^{(d)[2]}(x,u_1,u_2) = \mathcal{O}((k+h)y(x)^{k+d}).
\end{align*}
\end{lem}

\begin{proof}
The proof of the first statement is identical to the one of Lemma \ref{l:g2}, as we can derive identical recursive relations for $\gamma_{k,h}^{(d)[2u_1]}(x,u_1,u_2)$ and $\gamma_{k,h}^{(d)[2u_2]}(x,u_1,u_2)$ and a similar one for $\gamma_{k,h}^{(d)[u_1u_2]}(x,u_1,u_2)$:
\begin{align*}
 \gamma_{k+1,h}^{(d)[u_1u_2]}(x,u_1,u_2)&= y_{k+1,h}^{(d)}(x,u_1,u_2) \left(\sum\limits_{i\geq1} \frac{\partial}{\partial u_1} y_{k,h}^{(d)}(x^i,u_1^i,u_2^i)u_1^{i-1}\right)\left(\sum\limits_{i\geq1} \frac{\partial}{\partial u_2} y_{k,h}^{(d)}(x^i,u_1^i,u_2^i)u_2^{i-1}\right) \\
&+y_{k+1,h}^{(d)}(x,u_1,u_2)\sum\limits_{i\geq1}i \gamma_{k,h}^{(d)[u_1u_2]}(x^i,u_1^i,u_2^i)u_1^{(i-1)}u_2^{(i-1)}.
\end{align*}
We then prove the statement inductively with the following initial conditions (note therefore that
$\frac{\partial}{\partial s_i}Z_{n}(s_1,\ldots,s_n)=\frac{1}{i}Z_{n-i}(s_1,\ldots,s_{n-i})$ (cf.
\cite[Chapter 2, p. 25]{Da}):
\begin{align*}
\gamma_{0,h}^{(d)[2u_1]}(x,u_1,u_2)&=0,\\
\gamma_{0,h}^{(d)[2u_2]}(x,u_1,u_2)&\leq\gamma_h^{(d)[2]}(x,u_2)=\mc{O}(y(x)^{h+d}),\\
\gamma_{0,h}^{(d)[u_1u_2]}(x,u_1,u_2)&=x\frac{\partial}{\partial u_2}Z_{d-1}(y^{(d)}_h(x,u_2),\ldots,y_h^{(d)}(x^{d-1},u_2^{d-1}))\\
&=\sum_{r=1}^{d-1}\frac{\partial}{\partial s_r}Z_{d-1}(s_1,\ldots,s_{d-1})\Bigg|_{s_i=y_h(x^i,u_2^i)}\gamma_h^{(d)}(x^r,u_2^r)ru_2^{r-1}\\
&=\sum_{r=1}^{d-1}\frac{1}{r}Z_{d-r-1}(s_1,\ldots,s_{d-1-r})\Bigg|_{s_i=y_h(x^i,u_2^i)}\gamma_h^{(d)}(x^r,u_2^r)ru_2^{r-1}\\
&=\mc{O}\left(Z_{d-2}(y_h^{(d)}(x,u_2),\ldots,y_h^{(d)}(x^{d-2},u_2^{d-2}))\gamma_h^{(d)}(x,u_2)\right)\\
&=\mc{O}(y(x)^{h+2d-2}).
\end{align*}

For the proof of the second statement we define for $\ell \leq k$
\begin{align*}
 D_{\ell,h}^{(d)} = \sup_{\substack{x\in\Delta_\epsilon\\u_1\in\Xi_k,u_2\in\Xi_{k+h}}} \left|\frac{\gamma_{\ell,h}^{(d)[2]}(x,u,v)}{y(x)^{\ell+d}}\right|,
\end{align*}
as in the proof of the second part of Lemma \ref{l:g2}. We use the estimate 
\begin{align}\label{eq:ykh} \left|y_{\ell+1,h}^{(d)}(x,u_1,u_2)\right|\leq|y(x)|\exp\left(C_{\ell,h}^{(d)[u_1]}\frac{\eta_1}{k}+C_{\ell,h}^{(d)[u_2]}\frac{\eta_2}{k}+\mc{O}(L^\ell)\right),
\end{align}
which we obtained in the proof of Lemma \ref{lem:gamma2}.

From the recursive description, we can derive the following by applying known bounds from Lemma \ref{lem:gamma2} and from the previous statement, and from \eqref{eq:mult} and \eqref{eq:ykh}, similar to the proof of Lemma \ref{l:g2}.
\begin{align*}
  D_{\ell+1,h}^{(d)}&= \sup_{\substack{x\in\Delta_\epsilon\\u_1\in\Xi_k,u_2\in\Xi_{k+h}}}\left|\frac{\gamma_{\ell+1,h}^{(d)[2u_1]}(x,u_1,u_2)+\gamma_{\ell+1,h}^{(d)[2u_2]}(x,u_1,u_2)+\gamma_{\ell+1,h}^{(d)[u_1u_2]}(x,u_1u_2)}{y(x)^{k+d+1}}\right|\\
&=\sup_{\substack{x\in\Delta_\epsilon\\u_1\in\Xi_k,u_2\in\Xi_{k+h}}}\left|\frac{y_{\ell+1,h}^{(d)}(x,u_1,u_2)}{y(x)}\right|\\
&\times\left|\frac{\sum_{r=1}^2(\sum_{i\geq1}\gamma_{\ell,h}^{(d)[u_r]}(x^i,u_1^i,u_2^i)u_r^{i-1})^2+\prod_{r=1}^2(\sum_{i\geq1}\gamma_{\ell,h}^{(d)[u_r]}(x^i,u_1^i,u_2^i)u_r^{i-1})}{y(x)^{\ell+d}}\right.\\
&+\left.\frac{\sum_{i\geq1}\gamma_{\ell,h}^{(d)[2]}(x,u_1,u_2)+\sum_{r=1}^2\sum_{i\geq 2}(i-1)\gamma_{\ell,h}^{(d)[u_r]}(x^i,u_1^i,u_2^i)u_r^{i-2}}{y(x)^{\ell+d}}\right|\\
&\leq\exp\left(C_{\ell,h}^{(d)[u_1]}\frac{\eta_1}{k}+C_{\ell,h}^{(d)[u_2]}\frac{\eta_2}{k}+\mc{O}(L^\ell)\right)\\
&\times\left((C_{\ell,h}^{(d)[u_1]})^2y(x)^{k+d}+C_{\ell,h}^{(d)[u_1]}C_{\ell,h}^{(d)[u_2]}y(x)^{k+h+d}+(C_{\ell,h}^{(d)[u_2]})^2y(x)^{k+2h+d}+D_{\ell,h}^{(d)}+\mc{O}(L^\ell)\right)\\ &\leq D_{\ell,h}^{(d)}\exp\left(C_{\ell,h}^{(d)[u_1]}\frac{\eta_1}{k}+C_{\ell,h}^{(d)[u_2]}\frac{\eta_2}{k}\right)(1+\mc{O}(L^\ell))\\
&+\exp\left(C_{\ell,h}^{(d)[u_1]}\frac{\eta_1}{k}+C_{\ell,h}^{(d)[u_2]}\frac{\eta_2}{k}\right)\left((C_{\ell,h}^{(d)[u_1]})^2+C_{\ell,h}^{(d)[u_1]}C_{\ell,h}^{(d)[u_2]}+(C_{\ell,h}^{(d)[u_2]})^2+\mc{O}(L^\ell))\right)\\
&=D_{\ell,h}^{(d)}\alpha_{\ell,h}^{(d)}+\beta_{\ell,h}^{(d)}
\end{align*}
As in the proof of Lemma \ref{l:g2} we get
\begin{align}
 D_{k,h}^{(d)}&\leq \alpha_{0,h}D_{0,h}+\sum_{j=0}^{k-1}\beta_{j,h}^{(d)}\prod_{i=j+1}^{k-1}\alpha_{j,h}^{(d)}\\
&=\mc{O}(k)+\mc{O}(D_{0,h}^{(d)}).
\end{align}
It remains to prove that $D_{0,h}^{(d)}=\mc{O}(h)$:
\begin{align*}
\gamma_{0,h}^{(d)[2u_1]}(x,u_1,u_2)&=0\\
\gamma_{0,h}^{(d)[2u_2]}(x,u_1,u_2)&=\gamma_{h}^{(d)[2]}(x,u_2)+\frac{\partial^2}{\partial u_2^2} Z_{d-1}(y_h^{(d)}(x,u_2),\ldots, y_h^{(d)}(x^{d-1},u_2^{d-1}))\\ 
&=\gamma_{h}^{(d)[2]}(x,u_2)+\sum\limits_{l=1}^{d-1} \sum\limits_{j=1}^{d-l-1} Z_{d-j-l-1}( y_h^{(d)})\gamma_h^{(d)}(x^l,v^l) v^{l-1}\\
&\phantom{=\gamma_{h}^{(d)[2]}(x,u_2)+}+\sum\limits_{l=1}^{d-1}Z_{d-l-1}(y_h^{(d)})\gamma_h^{(d)[2]}(x^l,v^l)u^{l-1}\\
&\phantom{=\gamma_{h}^{(d)[2]}(x,u_2)+}+\sum\limits_{l=1}^{d-1}Z_{d-l-1}(y_h^{(d)})(l-1)\gamma_h^{(d)}(x^l,v^l)u^{l-2}\\
&=\mathcal{O}(h y(x)^{h+d})+\mc{O}(y(x)^{h+2d-3}+\mc{O}(hy(x)^{h+2d-2}+\mc{O}(y(x)^{h+2d-2})\\
&=\mc{O}(h y(x)^{h+d})
\end{align*} 
and 
\begin{align*} 
\gamma_{0,h}^{(d)[u_1u_2]}(x,u_1,u_2)&=\frac{\partial}{\partial u_2} Z_{d-1}(y_h^{(d)}(x,u_2),\ldots, y_h^{(d)}(x^{d-1},u_2^{d-1}))=\mc{O}(y(x)^{h+2d-2}) \\
 D_{0,h}^{(d)}&=\sup_{\substack{x\in\Delta_\epsilon\\u_1\in\Xi_k,u_2\in\Xi_{k+h}}}\left|\frac{\gamma_{0,h}^{(d)[2u_1]}(x,u_1,u_2)+\gamma_{0,h}^{(d)[2u_2]}(x,u_1,u_2)+\gamma_{0,h}^{(d)[u_1u_2]}(x,u_1,u_2)}{y(x)^{d}}\right|\\
&\leq\sup_{\substack{x\in\Delta_\epsilon\\u_1\in\Xi_k,u_2\in\Xi_{k+h}}}\mc{O}(hy(x)^h+y(x)^{h+2d-2})=\mc{O}(h)
\end{align*}
\end{proof}

\rem Obviously, for $x\in\Delta_\epsilon$ and $u_1\in\Xi_k,u_2\in\Xi_{k+h}$ the statement also holds for the partial derivatives:
\begin{align*}
 \gamma_{k,h}^{(d)[2u_1]}(x,u_1,u_2)&=\mc{O}(k y(x)^{k+d})\\
\gamma_{k,h}^{(d)[2u_2]}(x,u_1,u_2)&=\mc{O}((k+h)y(x)^{k+h+d})\\
\gamma_{k,h}^{(d)[u_1u_2]}(x,u_1,u_2)&=\mc{O}((k+h)y(x)^{k+d})
\end{align*}

\begin{lem}\label{lem:Ckmult}
For $x \in \Delta_\epsilon, u_1\in\Xi_{k}$ and $u_2\in\Xi_{k+h}$, with the same constants as in the previous lemmata, we can approximate
 \begin{align*}
 w_{k,h}^{(d)}(x,u_1,u_2) =& C_k^{(d)}(x)(u_1-1)y(x)^{k+d} + C_{k+h}^{(d)}(x)(u_2-1)y(x)^{k+h+d}\\
&+\mathcal{O}((k+h)y(x)^{k+d}(|u_1-1|^2 + |u_2-1|^2)).
\end{align*}
Furthermore
\begin{align*}
 \Sigma_{k,h}^{(d)}(x,u_1,u_2) =&\tilde{C}_k^{(d)}(x^2)(u_1-1)y(x^2)^{k+d} + \tilde{C}_{k+h}^{(d)} (x)(u_2-1)y(x^2)^{k+h+d}\\
&+\mathcal{O}(y(|x|^2)^{k}|u_1-1|^2 + y(|x|^2)^{k+h}|u_2-1|^2)
\end{align*}
\end{lem}
\begin{proof}
For the first statement, we expand $w_{k,h}^{(d)}(x,u_1,u_2)$ into a Taylor polynomial of degree $2$ around $u_1=u_2=1$ and obtain
\begin{align*}
 w_{k,h}^{(d)}(x,u_1,u_2)&=\gamma_k^{(d)}(x)(u_1-1)+\gamma_{k+h}^{(d)}(x)(u_2-1)+R\\
\textrm{with }\quad |R|&\leq \frac{1}{2}\left(\gamma_{k,h}^{(d)[2u_1]}(x,1+\vartheta_1(u_1-1),1+\vartheta_2(u_2-1))(u_1-1)^2\right.\\
&\phantom{\qquad + \frac{1}{2}{}+}+2\gamma_{k,h}^{(d)[u_1u_2]}(x,1+\vartheta_1(u_1-1),1+\vartheta_2(u_2-1))(u_1-1)(u_2-1)\\
&\phantom{\qquad + \frac{1}{2}{}+}\left.+\gamma_{k,h}^{(d)[2u_2]}(x,1+\vartheta_1(u_1-1),1+\vartheta_2(u_2-1))(u_2-1)^2\right).\end{align*}
Hence,
\begin{align*} w_{k,h}^{(d)}(x,u_1,u_2)&=C_k^{(d)}(x)(u_1-1)y(x)^{k+d} + C_{k+h}^{(d)}(x)(u_2-1)y(x)^{k+h+d}\\
&+\mathcal{O}((k+h)y(x)^{k+d}(|u_1-1|^2 + |u_2-1|^2)),
\end{align*}
where we can neglect the mixed derivatives as either $(u_1-1)^2$ or $(u_2-1)^2$ will determine the dominant part. For the second part we again use a Taylor polynomial, usin the fact that $|x^i|<\rho<1$ and $|u_r^i-1|\leq i|u_r-1|$ for $i >2, r=1,2$, hence the result follows immediately. 
\end{proof}
Note that the terms $u_1-1$ and $u_2-1$ are asymptotically proportional:  $\frac{u_2}{u_1} = \frac{e^{\frac{it_1}{\sqrt{n}}}-1}{e^{\frac{it_2}{\sqrt{n}}}-1}\sim\frac{t_2}{t_1}$, and that $y(x^2)^{k+h+d}$ is exponentially smaller than $y(x^2)^{k+d}$ as $h = \xi \sqrt{n}$.

\begin{lem}
 There exist constants such that $w_{k,h}^{(d)}=w_{k,h}^{(d)}(x,u_1,u_2)$ is given by 
\begin{align*}
 w_{k,h}^{(d)}= \frac{w_{0,h}^{(d)}y(x)^k}{1-f_k^{(d)} -\frac{w_{0,h}^{(d)}}{2}\frac{1-y(x)^k}{1-y(x)} + \mathcal{O}(|u_1-1| + |u_2-1|}
\end{align*}

for $u_1\in\Xi_k, u_2\in\Xi_{k+h}$ and $x\in\Delta_\epsilon$, where $f_k^{(d)}$ is given by:

\begin{align}
 f_k^{(d)}(x,u_1,u_2) = w_{0,h}^{(d)}(x,u_1,u_2)\sum\limits_{l=0}^{k-1} \frac{\Sigma_{l,h}(x,u_1,u_2)y(x)^{l+1}}{w_{l,h}^{(d)}(x,u_1,u_2)w_{l+1,h}(x,u_1,u_2)}.
\end{align}
\end{lem}
\begin{proof}
We can argue similarly as in the proof of Lemma \ref{lem:wkd} and derive the recursive description
\begin{align*}
 w_{k+1,h}^{(d)}=y w_{k,h}^{(d)} \left(1+\frac{\Sigma_{k,h}^{(d)}}{w_{k,h}^{(d)}}\right)\left(1+\frac{w_{k,h}^{(d)}}{2} + \mathcal{O}(w_{k,h}^{2(d)}) + \mathcal{O}(\Sigma_{k,h}^{(d)})\right),
\end{align*}

and equivalently
\begin{align*}
\frac{y}{w_{k+1,h}^{(d)}}\cdot\left(1+\frac{\Sigma_{k,h}^{(d)}}{w_{k,h}^{(d)}}\right)=\frac{1}{w_{k,h}^{(d)}}-\frac{1}{2} + \mathcal{O}(w_{k,h}^{(d)}) + \mathcal{O}\left(\frac{\Sigma_{k,h}^{(d)}}{w_{k,h}^{(d)}}\right).
\end{align*}

Further we get
\begin{align*}
\frac{y^{k+1}}{w_{k+1,h}^{(d)}}=\frac{y^{k}}{w_{k,h}^{(d)}}-\frac{\Sigma_{k,h}^{(d)}\cdot y(x)^{k+1}}{w_{k,h}^{(d)}w_{k+1,h}^{(d)}}-\frac{1}{2} y(x)^k+ \mathcal{O}(w_{k,h}^{(d)}y^k) + \mathcal{O}\left(\frac{\Sigma_{k,h}^{(d)}\cdot y^k}{w_{k,h}^{(d)}}\right).
\end{align*}

Solving the recurrence leads to 
\begin{align*}
 \frac{y^{k}}{w_{k,h}^{(d)}}&= \frac{1}{w_{0,h}^{(d)}} - \sum\limits_{l=0}^{k-1} \frac{\Sigma_{l,h}^{(d)}\cdot y(x)^{l+1}}{w_{l,h}^{(d)}w_{l+1,h}^{(d)}} - \frac{1}{2}\frac{1-y^k}{1-y} +  \mathcal{O}\big(\sum\limits_{l=0}^{k-1}\!\!\!\!\!\underbrace{w_{\ell,h}^{(d)}y^\ell}_{=\mathcal{O}(w_{0,h}y^{2\ell})}\!\!\!\!\!\big) + \mathcal{O}\bigg(\sum\limits_{l=0}^{k-1}\underbrace{\frac{\Sigma_{l,h}^{(d)}\cdot y^l}{w_{l,h}^{(d)}}}_{=\mathcal{O}(L^l)}\bigg)\\
&=\frac{1}{w_{0,h}^{(d)}}\left(1-\underbrace{w_{0,h}^{(d)}\sum\limits_{l=0}^{k-1} \frac{\Sigma_{l,h}^{(d)}y(x)^{l+1}}{w_{l,h}^{(d)}w_{l+1,h}^{(d)}}}_{=:f_k^{(d)}(x,u_1,u_2)} - \frac{w_{0,h}^{(d)}}{2}\frac{1-y^k}{1-y} +  \mathcal{O}(w_{0,h}^{2(d)}\frac{1-y^{2k}}{1-y^2}) + \mathcal{O}(w_{0,h}^{(d)}\frac{1-L^k}{1-L})\right).
\end{align*}

Observe that
\begin{align*}
 w_{0,h}^{(d)} &= y_{h}^{(d)}(x,v) +(u-1)xZ(S_{d-1},y_{h}^{(d)}(x,v), \ldots,y_{h}^{(d)}(x^i,v^i)) - y(x)\\
&=w_h^{(d)}+(u-1)xZ(S_{d-1},y_{h}^{(d)}(x,v), \ldots,y_{h}^{(d)}(x^i,v^i))\\
&=C_h^{(d)}(x)(u_2-1)y(x)^{h+d}+(u_1-1)y(x)^d=\mc{O}(|u_1-1|+|u_2-1|)
\end{align*}
\end{proof}

In the following, we denote by $U:=(u_1-1)y(x)^d$ and $W:=w_h^{(d)}(x,u_2)$. Note that $w_{0,h}\sim U+W$. By Lemma \ref{lem:Ckmult} we obtain for $w_{\ell,h}^{(d)}$ (note that $C_{\ell+h}^{(d)}(x)=C_h^{(d)}(x)(1+L^\ell)$)
\begin{align*}
 w_{\ell,h}^{(d)}(x,u_1,u_2)&=C_\ell^{(d)}(x)(u_1-1)y(x)^{\ell+d}+C_{\ell+h}^{(d)}(x)(u_2-1)y(x)^{\ell+h+d}\\
&\qquad +\mc{O}\left((\ell+h)y(x)^{\ell}(|u_1-1|^2+|u_2-1|^2)\right)\\
&=y(x)^{\ell}\left(C_\ell^{(d)}U + C_{h}^{(d)}(x)(u_2-1)y(x)^{h+d}\right.\\
&\qquad \left.+\mc{O}\left((\ell+h)(|u_1-1|^2+|u_2-1|^2)\right)\right)\\
&=y(x)^\ell(C_\ell^{(d)}(x)U+W)(1+\mc{O}(h(|u_1-1|+|u_2-1|))
\end{align*}

We use the representation \eqref{eq:tildeC} for $\tilde{C}_\ell(x)$, which we already used in the proof of Lemma \ref{lem:wkd}, and omit all error terms, to obtain by telescoping
\begin{align*}
 f_k^{(d)}&(x,u_1,u_2) = w_{0,h}^{(d)}(x,u_1,u_2)\sum_{\ell=0}^{k-1}\frac{y(x)^{\ell+1}\left(\tilde{C}_\ell^{(d)}(x)(u_1-1)y(x^2)^{k+d}\right)}{y(x)^\ell(C_\ell^{(d)}(x)U+W)y(x)^{\ell+1}(C_{\ell+1}^{(d)}(x)U+W)}\\
&= Uw_{0,h}^{(d)}(x,u_1,u_2) \sum\limits_{l=0}^{k-1} \frac{\tilde{C}_l^{(d)}(x)\left(\frac{y(x^2)}{y(x)}\right)^{l+d}}{(C_l^{(d)}(x)U+W)(C_{l+1}^{(d)}(x)U+W)}\\
&= w_{0,h}^{(d)}(x,u_1,u_2)\sum\limits_{l=0}^{k-1} \frac{(C_{l+1}^{(d)}(x)U+W) - (C_l^{(d)}(x)U+W)}{(C_l^{(d)}(x)U+W)(C_{l+1}^{(d)}(x)U+W)}\\
&=w_{0,h}^{(d)}(x,u_1,u_2)\left(\frac{1}{C_0^{(d)}(x)U + W} - \frac{1}{C_k^{(d)}(x)U+ W}\right).
\end{align*}
As we know from \eqref{eq:C0}, $C_0^{(d)} = \frac{x Z(S_{d-1})}{y(x)^d} =\mc{O}(1)$ near $u=1$ (analytic), hence
\begin{align*}
f_k^{(d)}(x,u,v)  \sim \left(1-\frac{(U+W)}{C_k^{(d)}(x)U+ W}\right).
\end{align*}

Using
\begin{align*}
C_k^{(d)}(x) &\sim C \rho^d,\\
(u_1-1)&\sim \frac{it_1}{\sqrt{n}},\\
y(x)^k &\sim e^{-\kappa b \sqrt{-\rho s}},\\
1-y(x)&\sim b \sqrt{\frac{\rho s}{n}},
\end{align*}
and $w_{0,h}^{(d)}(x,u_1,u_2)\sim W+U$, we can derive  
\begin{align*}
 w_{k,h}^{(d)} &= \frac{w_{0,h}^{(d)}y(x)^k}{\frac{((U+W)}{C \rho^d U+W} - \frac{w_{0,h}^{(d)}}{2} \frac{1-y(x)^k}{1-y(x)}}\\
&=\frac{(C\rho^d\frac{it_1}{\sqrt{n}} + w_h^{(d)}(x,v))\sqrt{-s} e^{-\kappa b \sqrt{-\rho s}}}{\sqrt{-s} - ((C\rho^d\frac{it_1}{\sqrt{n}} + w_h^{(d)}(x,v))\frac{1}{2b\sqrt{\frac{\rho}{n}}}(1-e^{-\kappa b \sqrt{-\rho s}})}\\
&=\frac{C\rho^d}{\sqrt{n}}\frac{(it_1 + w_h^{(d)}(x,v))\sqrt{-s} e^{-\frac{\kappa}{2} b \sqrt{-\rho s}}}{\sqrt{-s}e^{\frac{\kappa}{2} b \sqrt{-\rho s}}- ((C\rho^d\frac{it_1}{\sqrt{n}} + w_h^{(d)}(x,v))\frac{1}{b\sqrt{\frac{\rho}{n}}}(\sinh(\frac{\kappa}{2} b \sqrt{-\rho s}))}
\end{align*}
and with the expansion \eqref{eq:wkd} of $w_h^{(d)}(x,u_2)$ with $u_2=e^{\frac{it_2}{\sqrt{n}}}$ and $h=\xi\sqrt{n}$, given by Theorem \ref{th:exp}, we can derive the expansion given in Theorem \ref{th:expmult}. 
\end{proof}

\begin{proof}[Proof of Theorem \ref{th:3}]
The characteristic function of the two dimensional distribution is given by

\begin{align}\label{eq:phi2}
\phi_{k,k+h,n}^{(d)}(t_1,t_2) &= \frac{1}{y_n} [x^n] y_{k,h}^{(d)}(x,e^{\frac{it_1}{\sqrt{n}}},e^{\frac{it_2}{\sqrt{n}}}) \nonumber\\
&=\frac{1}{2 \pi i y_n} \int_\Gamma y_{k,h}^{(d)}(x,e^{\frac{it_1}{\sqrt{n}}},e^{\frac{it_2}{\sqrt{n}}}) \frac{dx}{x^{n+1}}\nonumber\\
&=1+\frac{1}{2 \pi i y_n} \int_\Gamma w_{k,h}^{(d)}(x,e^{\frac{it_1}{\sqrt{n}}},e^{\frac{it_2}{\sqrt{n}}}) \frac{dx}{x^{n+1}}.
\end{align}

We use the same contour as in the one dimensional case. With the same arguments, only integration over $\gamma$ contributes to the result, hence the representation \eqref{eq:wkdmult} of $w_{k,h}^{(d)}$ leads to:

\begin{align*}
&\phi_{k,h,n}^{(d)}(t_1,t_2) 
=1+\frac{\sqrt{2}}{\sqrt{\pi}i}\\ &\times\int\limits_{1-i \log^2 n}^{1+i \log^2 n} \frac{\frac{C \rho^d}{b\sqrt{2\rho}}i\left( t_1 +\frac{t_2 \sqrt{-s} e^{(-\frac{1}{2}\xi b \sqrt{-\rho s })}}{\sqrt{-s}e^{(\frac{1}{2}\xi b \sqrt{-\rho s })}-\frac{it_2 C \rho^d}{\sqrt{\rho}b}\sinh{(\frac{1}{2}\xi b \sqrt{-\rho s})})}\right)\sqrt{-s} e^{(-\frac{1}{2}\kappa b \sqrt{-\rho s })}}{\sqrt{-s}e^{(-\frac{\kappa}{2} b \sqrt{-\rho s })}-i\frac{C\rho^d}{2b\sqrt{\rho}}\left(t_1 +\frac{t_2 \sqrt{-s} e^{(-\frac{1}{2}\xi b \sqrt{-\rho s })}}{\sqrt{-s}e^{(\frac{1}{2}\xi b \sqrt{-\rho s })}-\frac{it_2 C \rho^d}{\sqrt{\rho}b}\sinh{(\frac{1}{2}\xi b \sqrt{-\rho s})})}\right)(\sinh{(\frac{\kappa}{2} b \sqrt{-\rho s})})} \\
&\xrightarrow{n\to\infty}\psi_{\kappa,\xi}(t_1,t_2)
\end{align*} 
where $\psi_{\kappa,\xi}(t_1,t_2)$ is the characteristic function of the random variable
$\frac{C_d\rho^d}{\sqrt{2\rho}b} \(l\(\frac{b\sqrt{\rho}}{2\sqrt{2}}\kappa,
\frac{b\sqrt{\rho}}{2\sqrt{2}}\xi\)\)$.
\end{proof}

\section{Tightness}\label{sec:6}

We must show the estimate \eqref{tightness} in Theorem~\ref{thm:tightness}. The fourth moment in
\eqref{tightness} can be obtained by by applying the operator
$\(u\pdiff{}u{}\)^4$ and setting $u=1$ afterwards. Hence, using the transfer lemma of Flajolet and
Odlyzko \cite{FO90} it turns out that it suffices to show that
\begin{equation}\label{eqtoshow}
\U{\left( \pdiff{}u{}+7\pdiff{}u2+6\pdiff{}u3+\pdiff{}u4 \right)
\tilde y_{r,h}\(x,u,u^{-1}\)} = \Ord{ \frac{h^2}{1-|y(x)|} }
\end{equation}
uniformly for $x\in \Delta$ and $h\ge 1$ (see \cite[pp.2046]{DrGi} for the detailed argument).

Set
\[
\gamma_k^{(d)[j]}(x) = \U{\frac{\partial^j y_k(x,u)}{\partial u^j} }
\qquad \mbox{and}\qquad
\gamma_{k,h}^{(d)[j]}(x) = \U{\frac{\partial^j \tilde y_{r,h}\(x,u,\frac 1u\)}
{\partial u^j}}.
\]

The left-hand side of \eqref{eqtoshow} is a linear combination of $\gamma_{k,h}^{(d)[j]}(x)$ for
$j=1,2,3,4.$ Therefore we need bound for those quantities. We will derive upper bounds for all $j$
since this more general result is easier to achieve. We start with an auxiliary result. 

\bl\label{onelevel}
Let $j$ be a positive integer. 
Under the assumption that for all $i\le j$ the bound 
$\gamma_k^{(d)[i]}(x) =\Ord{|x/\rho|^k}$ holds uniformly for $|x|\le \rho$, we have $\U{\(\pdiff{}u{}\)^j \d\Sigma}=\Ord{L^k}$ for some positive constant $L<1$.
\el

\bpf
By Fa\`a di Bruno's formula we have 
\begin{align*} 
\U{\(\pdiff{}u{}\)^j \d\Sigma} &= \sum_{i\ge 2}\frac 1i \U{\(\pdiff{}u{}\)^j \d w (x^i,u^i)} \\
&= \sum_{i\ge 2}\frac 1i 
\sum_{\sum_{m=1}^j m \nu_m=j} \frac{j!}{\nu_1!\cdots \nu_j!} \gamma_k^{(d)[\nu_1+\cdots+\nu_m]}
(x^i,1) \prod_{\lambda=1}^j \(\frac1{\lambda!}\U{\(\pdiff{}u{}\)^\lambda u^i}\)^{\nu_\lambda}.
\end{align*} 
By our assumption we have $\gamma_k^{(d)[\nu_1+\cdots+\nu_m]}(x^i,1)=\Ord{|x^i/\rho|^k}$. The
product is essentially a derivative of order $j=\sum \lambda \nu_\lambda$ of $u^i$ and can
therefore be estimate by $\Ord{i^j}$. So the whole expression is bounded by a constant times 
$\sum_{i\ge 2} i^{j-1} x^{ik}/\rho^k=\Ord{(|x^2|/\rho)^ki}=\Ord{(\rho+\eps)^k}$. Hence we can
choose $L=\rho+\eps$ to get the desired bound. 
\epf

Exactly the same line of arguments yield the analogous result for two levels:

\bl\label{morelevels}
Let $j$ be a positive integer and set 
\begin{equation} \label{sigmaschlange}
\tilde\Sigma_{k,h}^{(d)}= \sum_{i\ge 2}\frac 1i w_{k,h}^{(d)}(x^i,u^i,u^{-i}). 
\end{equation} 
Under the assumption that for all $i\le j$ the bound 
$\gamma_{k,h}^{(d)[i]}(x) =\Ord{|x/\rho|^k}$ holds uniformly for $|x|\le \rho$ we have 
$\U{\(\pdiff{}u{}\)^j \tilde\Sigma_{k,h}^{(d)}}=\Ord{L^k}$ for some positive constant $L<1$.
\el

With the auxiliary lemmas we can easily get bounds for $\gamma_k^{(d)[j]}(x)$ and
$\gamma_{k,h}^{(d)[j]}(x)$. 

\bl
We have 
\begin{equation} \label{g:1}
\gamma_k^{(d)[1]}(x)=\left\{ \begin{array}{ll}
\Ord{1} & \mbox{uniformly for $x\in \Delta$}, \\
 \Ord{|x/\rho|^k} & \mbox{uniformly for $|x|\le \rho$}
\end{array}\right.
\end{equation}
and for $\ell>1$
\begin{equation} \label{g:2}
\gamma_k^{(d)[\ell]}(x)=\left\{ \begin{array}{ll}
\Ord{\min\(k^{\ell-1}, \frac {k^{\ell-2}}{1-|y(x)|}\)} & \mbox{uniformly for $x\in \Delta$}, \\
 \Ord{|x/\rho|^k} & \mbox{uniformly for $|x|\le \rho$}
\end{array}\right.
\end{equation}
\el

\bpf
The estimate \eqref{g:1} essentially follows from Lemma~\ref{l:g}: We know 
$\gamma_k^{(d)[1]}(x)=C^{(d)}(x)y(x)^{k+d}(1+\Ord{L^k})=\Ord{1}$ with some $0<L<1$ and $|y(x)|\le
1$ and this is sufficient to show the first part of \eqref{g:1}. 

If $|x|\le \rho$ we can exploit the convexity of $y(x)$ on the positive real line to get
$|y(x)|\le |x/\rho|$. This implies $\gamma_k^{(d)[1]}(x)=\Ord{|x/\rho|^{k+d}}$, an even better
bound than stated in the assertion. 

Now we are left with the induction step. Again we use Fa\`a di Bruno's formula and the fact 
that $\d w(x,1)=\d\Sigma(x,1)=0$ and obtain 
\begin{align} 
\gamma_k^{(d)[\ell]}(x)&=\U{\pdiff{}u{} \d w(x,u)}=y(x) \U{\pdiff{}u{}
\exp\(w_{k-1}^{(d)}(x,u)+\Sigma_{k-1}^{(d)}(x,u)\)} \nonumber \\
&=
\sum_{\sum_{i=1}^\ell i\lambda_i=\ell} \frac{\ell!}{\lambda_1!\cdots \lambda_\ell!} 
\prod_{j=1}^{\ell-1}\(\frac1{j!} 
\U{\(\pdiff{}u{}\)^j \(w_{k-1}^{(d)}(x,u)+\Sigma_{k-1}^{(d)}(x,u)\)}\)^{\lambda_j} \nonumber \\
&\quad +y(x)\U{\(\pdiff{}u{}\)^\ell \(w_{k-1}^{(d)}(x,u)+\Sigma_{k-1}^{(d)}(x,u)\)} \nonumber \\
&=\sum_{\sum_{i=1}^\ell i\lambda_i=\ell} \frac{\ell!}{\lambda_1!\cdots \lambda_\ell!}
\prod_{j=1}^{\ell-1}\(\frac{\gamma_{k-1}^{(d)[j]}(x)+\Gamma_{k-1}^{(d)[j]}(x)}{j!}
\)^{\lambda_j} + y(x)(\gamma_{k-1}^{(d)[\ell]}(x)+\Gamma_{k-1}^{(d)[\ell]}) \label{FdB}
\end{align} 
where $\Gamma_{k-1}^{(d)[\ell]}=\U{\(\pdiff{}u{}\)^\ell\Sigma_{k-1}^{(d)}(x,u)}$. 

Consider the case $|x|\le \rho$. The product comprises only terms which essentially have the form
$\gamma_{k-1}^{(d)[j]}(x)+\Gamma_{k-1}^{(d)[j]}(x)$ with $j<\ell$. Thus by the induction 
hypothesis, $\gamma_{k-1}^{(d)[j]}(x)=\Ord{|x/\rho|^j}$. Therefore the assumption of
Lemma~\ref{onelevel} is satisfied and the terms as a whole are bounded by $C\cdot |x/\rho|^j$. 
Since $\sum_{j=1}^{\ell-1} j\lambda_j=\ell$ we get 
$$
\gamma_k^{(d)[\ell]}(x)= y(x)(\gamma_{k-1}^{(d)[\ell]}(x)+\Gamma_{k-1}^{(d)[\ell]}+
\Ord{|x/\rho|^\ell}. 
$$
So we finally get the desired estimate by induction on $k$ and Lemma~\ref{onelevel}, starting with 
\begin{equation} \label{a_null}
\gamma_0^{(d)[\ell]}=\begin{cases}
xZ_{d-1}(y(x),y(x^2),\dots,y(x^{d-1})) & \mbox{ if } \ell=1, \\
0 & \mbox{ else.}
\end{cases}
\end{equation} 

Now let us turn to general $x\in\Delta$. Like before we focus first on the terms of the product of
\eqref{FdB}. Again the induction hypothesis guarantees that the assumption of Lemma~\ref{onelevel}
is satisfied and so $\Gamma_{k-1}^{(d)[j]}(x)$ is exponentially small. Furthermore, the induction
hypothesis implies $\gamma_{k-1}^{(d)[j]}(x)=\Ord{\min\(k^{j-1}, \frac {k^{j-2}}{1-|y(x)|}\)}$.
Since $\gamma_{k-1}^{(d)[1]}(x)=\Ord{1}$ this implies
\begin{equation} \label{prodbound}
\prod_{j=1}^{\ell-1}\(\frac{\gamma_{k-1}^{(d)[j]}(x)+\Gamma_{k-1}^{(d)[j]}(x)}{j!}
\)^{\lambda_j}
=\Ord{\min\(k^{\sum_{j=1}^{\ell-1}(j-1)\lambda_j}, 
\frac
{k^{\sum_{j=2}^{\ell-1}(j-2)\lambda_j}}{(1-|y(x)|)^{\sum_{j=2}^{\ell-1}\lambda_j}
}\)}.
\end{equation} 
Set 
$$
A=k^{\sum_{j=1}^{\ell-1}(j-1)\lambda_j} \mbox{ and }
B=\frac{k^{\sum_{j=2}^{\ell-1}(j-2)\lambda_j}}{(1-|y(x)|)^{\sum_{j=2}^{\ell-1}\lambda_j}}.
$$
Note that $\sum_{j=1}^{\ell-1}(j-1)\lambda_j=\ell-\sum_{j=1}^{\ell-1}\lambda_j$. Since the term
correspondung to $\lambda_\ell=0$ in Fa\`a di Bruno formula is the very last term in \eqref{FdB},
we must have $\sum_{j=1}^{\ell-1}\lambda_j\ge 2$ and thus $A\le k^{\ell-2}$. Moreover, we have 
$$
\sum_{j=2}^{\ell-1}(j-2)\lambda_j=\ell-k_1-2\sum_{j=2}^{\ell-1} k_j\le \ell-3
$$
since $\sum_{j=2}^{\ell-1}\lambda_j<2$ implies $k_1>0$ and, in particular,
$\sum_{j=2}^{\ell-1}\lambda_j=2$ implies $k_1=\ell$.
Therefore 
$$
B\le \frac{k^{\ell-3}}{(1-|y(x)|)^{\sum_{j=2}^{\ell-1}\lambda_j}}.
$$
We want to show that 
\begin{equation} \label{desiredbound}
B\le \frac{k^{\ell-3}}{1-|y(x)|}. 
\end{equation} 
Set 
$A_j=k^{j-1}$ and $B_j=k^{j-2}/(1-|y(x)|)$. Note that $B_j<A_j$ is equivalent to $1/(1-|y(x)|)<k$.
Therefore the term $B$ appears in our upper bound \eqref{prodbound} if and only if $x$ is such
that $1/(1-|y(x)|)<k$. But this implies that $B\le \frac{k^{\ell-3}}{1-|y(x)|}$ as desired, because
the desired bound is equivalent to 
$$
\frac1{1-|y(x)|}<k^{\(-2+2\sum_{j=1}^{\ell-1}\lambda_j-\lambda_1\)/\(-1+\sum_{j=2}^{\ell-1}\lambda_j\)}=
k^{1+\alpha} 
$$
where $\alpha=\(\sum_{j=1}^{\ell-1}\lambda_j-1\)\left/\(-1+\sum_{j=2}^{\ell-1}\lambda_j\)\right.
>0$ and hence the desired bound \eqref{desiredbound}
is weaker than $1/(1-|y(x)|)<k$. 

Now let $a_k:=\gamma_k^{(d)[\ell]}(x)$. We have shown so far that 
$$
a_k=y(x)a_{k-1} +y(x)A_k \mbox{ with } A_k=\Ord{\min\(k^{\ell-2}, \frac {k^{\ell-3}}{1-|y(x)|}\)}
$$
and we know that $a_0$ is given by \eqref{a_null}. Solving this recurrence relation gives 
$$
a_k=y(x)^ka_0 +\Ord{\btr{y(x)\frac{1-y(x)^k}{1-y(x)}}\cdot \min\(k^{\ell-2}, \frac
{k^{\ell-3}}{1-|y(x)|}\)}.
$$
Since $\btr{y(x)\frac{1-y(x)^k}{1-y(x)}}\le k$ and $a_0y(x)^k=\Ord{y(x)^{k+d}}=\Ord{1}$ we get the
desired bound for $a_k$ and the proof is complete.
\epf

\bl \label{lastlemma}
We have 
\begin{equation} \label{g:3}
\gamma_{k,h}^{(d)[1]}(x)=\left\{ \begin{array}{ll}
\Ord{1} & \mbox{uniformly for $x\in \Delta$}, \\
 \Ord{|x/\rho|^k} & \mbox{uniformly for $|x|\le \rho$,}
\end{array}\right.
\end{equation}
and for $\ell>1$
$$ 
\gamma_{k,h}^{(d)[\ell]}(x)=\left\{ \begin{array}{ll}
\Ord{\min\(k^{\ell-1}, \frac {k^{\ell-2}}{1-|y(x)|}\)} & \mbox{uniformly for $x\in \Delta$}, \\
 \Ord{|x/\rho|^k} & \mbox{uniformly for $|x|\le \rho$} 
\end{array}\right. 
$$
\el

\bpf
As the bounds are precisely the same as in the previous lemma, the induction step works in an analogous
way, using Lemma~\ref{morelevels} instead of Lemma~\ref{onelevel}. Thus we only have to show the initial step of the induction, Eq.~\eqref{g:3}. 

We can use a similar reasoning as in the proof of \cite[Lemma~7]{DrGi}. Indeed, by applying the 
operator $\U{\pdiff{}u{}\; \cdot\;}$ to \eqref{eq:rec2dim} we obtain the recurrence relation 
\[
\gamma_{k+1,h}^{(d)[1]}(x) = y(x) \sum_{i\ge 1} \gamma_{k,h}^{(d)[1]}(x^i)
\]
with initial value $\gamma_{0,h}^{(d)[1]}(x) = xZ_{d-1}(y(x),y(x^2),\dots,y(x^{d-1})) -
\gamma_h(x)$. Induction on $k$ gives the representation 
$\gamma_{k,h}^{(d)[1]}(x) = \gamma_k^{(d)[1]}(x) - \gamma_{k+h}^{(d)[1]}(x)$ and using
$\gamma_k^{(d)[1]}(x)=C^{(d)}(x)y(x)^{k+d}(1+\Ord{L^k})$ from Lemma~\ref{l:g} we obtain 
$$
\gamma_k^{(d)[1]}(x)=\Ord{\sup_{x\in \Delta} |y(x)^{k+d}(1-y(x)^h)| + L^k}=\Ord{\frac{h}{k+d+h}}.
$$
Since the last term is bounded, the proof is complete.
\epf

Now, applying Lemma~\ref{lastlemma} to \eqref{eqtoshow} proves tightness and
Theorem~\ref{thm:tightness} after all.

\section{The Joint Distribution of Two Degrees}
We want to gain knowledge on the correlation between two different degrees $d_1, d_2$ in a certain level $k=\kappa\sqrt{n}$. 

\subsection{The covariance $\mathbb{C}\textrm{ov}(X_n^{(d_1)}(k),X_n^{(d_2)}(k))$}

The covariance \index{covariance} of two random variables $X$ and $Y$ is given by 
\[\mathbb{C}\mathrm{ov}(X,Y)=\mathbb{E}(XY)-\mathbb{E}(X)\mathbb{E}(Y).\]
In this section, we will prove the result on the covariance function of the two random variables $X_n^{(d_1)}(k)$ and $X_n^{(d_2)}(k)$, counting the vertices of degree $d_1$ and $d_2$, respectively, on level $k$, given in Proposition~\ref{prop:covariance}.

To compute $\mathbb{E}\left(X_n^{(d_1)}(k)\cdot X_n^{(d_2)}(k))\right)$ we need to determine
\[\frac{1}{y_n}[x^n]\left[\frac{\partial^2}{\partial u\partial v}y_k(x,u,v)\right]_{u=v=1},\]
while $\mathbb{E}(X_n^{(d_1)}(k))$ and $\mathbb{E}(X_n^{(d_2)}(k))$ are given by
\[\frac{1}{y_n}[x^n]\gamma_k^{(d_1)}(x) \textrm{ and } \frac{1}{y_n}[x^n]\gamma_k^{(d_2)}(x),\textrm{respectively.}\]

We use the notations 
\[
\gamma_k^{(d_1)}(x,u,v)=\diff{u}y_k(x,u,v),\quad 
\gamma_k^{(d_2)}(x,u,v)=\diff{v}y_k(x,u,v),\]
as well as (recall~\ref{eq:tildeg}
\[
\tilde{\gamma_k}^{(d_1,d_2)}(x,u,v)=\frac{\partial^2}{\partial uv}y_k(x,u,v)
\textrm{ and
} 
\tilde{\gamma_k}^{(d_1,d_2)}(x)=\tilde{\gamma_k}^{(d_1,d_2)}(x,1,1).\]

\begin{lem}\label{lem:covrep}
 There exist constants $\epsilon$ and $\theta$ such that for $z\in\Delta(\eta,\theta)$
\[\tilde{\gamma}_k^{(d_1,d_2)}(x)=C^{(d_1)}(x)\cdot C^{(d_2)}(x)y(x)^{k+d_1+d_2}\sum_{\ell=0}^{k-1}(y(x)^\ell+\mc{O}(L^\ell)),\]
where $C^{(d_1)}(x)$ and $C^{(d_2)}(x)$ are given in Lemma~\ref{l:g}.
\end{lem}

\begin{proof}
We use the recursive representation \eqref{eq:recgamma} for $\gamma^{(d_1)}(x,u,v)$ with the additional variable $v$. This gives
\[\gamma^{(d_1)}_{k+1}(x,u,v)=y_{k+1}^{(d)}(x,u,v)\sum\limits_{i\geq 1} \gamma_k^{(d_1)}(x^i,u^i,v^i)u^{i-1}.\]
Derivating with respect to $v$ gives
\begin{align*}
 \tilde{\gamma}_{k+1}^{(d_1,d_2)}(x,u,v)&=\gamma^{(d_2)}(x,u,v)\sum\limits_{i\geq 1} \gamma_k^{(d_1)}(x^i,u^i,v^i)u^{i-1}\\
&\phantom{=}+y_{k+1}(x,u,v)\sum\limits_{i\geq 1}i \tilde{\gamma}_k^{(d_1,d_2)}(x^i,u^i,v^i)u^{i-1}v^{i-1}\\
&=y_{k+1}(x,u,v)\left(\sum\limits_{i\geq 1} \gamma_k^{(d_1)}(x^i,u^i,v^i)u^{i-1}\right)\left(\sum\limits_{i\geq 1} \gamma_k^{(d_2)}(x^i,u^i,v^i)v^{i-1}\right)\\
&\phantom{=}+y_{k+1}(x,u,v)\sum\limits_{i\geq 1}i \tilde{\gamma}_k^{(d_1,d_2)}(x^i,u^i,v^i)u^{i-1}v^{i-1},\\
\end{align*}
with $\tilde{\gamma}^{(d_1,d_2)}_0(x)=0$. Setting $u=v=1$ we obtain
\begin{align*}
 \tilde{\gamma_{k+1}}^{(d_1,d_2)}(x)=y(x)\left[\left(\sum\limits_{i\geq 1}
\gamma_k^{(d_1)}(x^i)\right)\left(\sum\limits_{i\geq 1}
\gamma_k^{(d_2)}(x^i)\right)+\sum\limits_{i\geq 1}i \tilde{\gamma}_k^{(d_1,d_2)}(x^i)\right]\\
=y(x)\left((\gamma_k^{(d_1)}(x)+\Gamma_k^{(d_1)}(x))(\gamma_k^{(d_2)}(x)+\Gamma_k^{(d_2)}(x))+\tilde{\Gamma}^{(d_1,d_2)}_k(x)+\tilde{\gamma_{k}}^{(d_1,d_2)}(x)\right),
\end{align*}
where we use the notations $\Gamma_k^{(d_1)}(x)=\sum\limits_{i\geq 2} \gamma_k^{(d_1)}(x^i)$ and $\Gamma_k^{(d_2)}(x)=\sum\limits_{i\geq 2} \gamma_k^{(d_2)}(x^i)$ as in the proof of Lemma~\ref{l:g}, and $\tilde{\Gamma}_k^{(d_1,d_2)}(x)=\sum\limits_{i\geq 2} i\tilde{\gamma}_k^{(d_1,d_2)}(x^i)$. Solving the recurrence, we get
\begin{equation}\label{eq:tigammaexpl}
 \tilde{\gamma}^{(d_1,d_2)}_k(x)=\sum_{\ell=1}^{k-1}y(x)^{k-\ell}\left((\gamma_\ell^{(d_1)}(x)+\Gamma_\ell^{(d_1)}(x))(\gamma_\ell^{(d_2)}(x)+\Gamma_\ell^{(d_2)}(x))+\tilde{\Gamma}^{(d_1,d_2)}_\ell(x)\right)
\end{equation}
From Corollary~\ref{cor:sumgamma} we know that $\Gamma_\ell^{(d_1)}(x)=\mc{O}(L^\ell)$ and
$\Gamma_\ell^{(d_2)}(x)=\mc{O}(L^\ell)$ in $\Theta$ for $u=v=1$. Together with
Equation~\eqref{eq:tildegamma} we have
\begin{align*}
 \tilde{\gamma}^{(d_1d_2)}_k(x)=\sum_{\ell=1}^{k-1}y(x)^{k-\ell}\left((C^{(d_1)}(x)y(x)^{\ell+d_1}+\mc{O}(L^\ell))(C^{(d_2)}y(x)^{\ell+d_2}+\mc{O}(L^\ell))+\mc{O}(L^\ell)\right),
\end{align*}and the result follows.
\end{proof}

To extract coefficients we will use Cauchy's formula.
\[[x^n]\tilde{\gamma}^{(d_1,d_2)}(x)=\frac{1}{2\pi
i}\int_{\delta}\tilde{\gamma}^{(d_1,d_2)}(x)\frac{1}{z^{n+1}}\mathrm{d}x,\]
where $\delta$ is the truncated contour $\delta=\delta_1\cup\delta_2\cup\delta_3\cup\delta_4$ given by
\begin{align}\label{eq:delta}
\begin{split}
\delta_1&=\left\{x=a+\frac{\rho i}{n}\bigg|\rho\leq a \leq \rho+\eta\log^2n/n\right\},\\
\delta_2&=\left\{x=\rho\left(1-\frac{e^{i\varphi}}{n}\right)\bigg|-\frac{\pi}2\leq\varphi\leq\frac{\pi}2\right\},\\
\delta_3&=\left\{x=a-\frac{\rho i}{n}\bigg|\rho\leq a \leq \rho+\eta\log^2n/n\right\}
\end{split}
\end{align}
and $\delta_4$ being a circular arc closing the contour, cf. Figure~\ref{fig:delta}.
\begin{figure}[ht]
\centering
\includegraphics[width=0.4\textwidth]{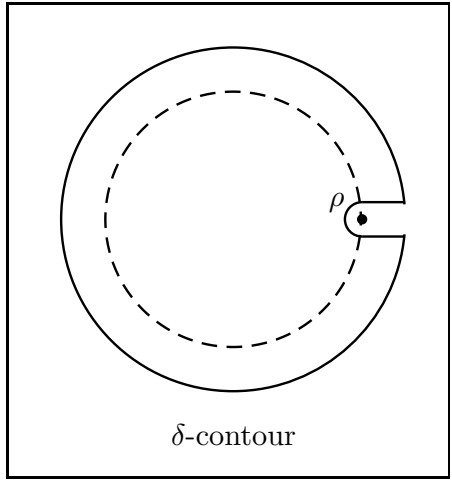}
\caption{The integration contour $\delta$}\label{fig:delta}
\end{figure}

It can be shown that the contribution of the circular arc $\delta_4$ is exponentially small and
thus asymptotically negligible. Near $\rho$, more
precicely for $z=\rho(1+\frac{s}{n})\in \delta_1\cup\delta_2\cup\delta_3$ and $k=\kappa\sqrt{n}$, we have
\begin{align*}
 y(x)^{d_1+d_2}&=1+\mc{O}\left(\sqrt{\left|\frac{s}{n}\right|}\right), \\
1-y(x)&\sim
b\sqrt{\rho}\sqrt{-\frac{s}{n}}\left(1+\mc{O}\left(\sqrt{\left|\frac{s}{n}\right|}\right)\right), \\
y(x)^k&\sim\exp(-\kappa b\sqrt{-\rho
s})\left(1+\mc{O}\left(\left|\frac{s}{\sqrt{n}}\right|\right)\right), \\
C^{(d_1)}(x)&\sim C_{d_1}\rho^{d_1}+\mc{O}\left(\left|\frac{s}{n}\right|\right),\qquad C^{(d_2)}(x)\sim
C_{d_2}\rho^{d_2}+\mc{O}\left(\left|\frac{s}{n}\right|\right).
\end{align*} 
Hence, the expected value $[x^n]\tilde{\gamma}^{(d_1,d_2)}(x)$ is given by 
\[[x^n]\tilde{\gamma}^{(d_1,d_2)}(x)\sim C_{d_1}C_{d_2}\rho^{d_1+d_2}\frac{1}{2\pi
i}\int_{\delta}\frac{\sqrt{n}}{b\sqrt{\rho}\sqrt{-s}}e^{-\kappa b\sqrt{-\rho s}}(1-e^{-\kappa
b\sqrt{-\rho s}})e^{-s}\frac{1}{n}\rho^{-n}\textrm{d}s,\]
as $\sum_{\ell=0}^{k-1}y(x)^\ell=\frac{1-y(x)^k}{1-y(x)}$. 
Integrals of the shape $\int_\delta (-s)^\mu e^{-\alpha\sqrt{-s}-s}\mathrm{d}s$ can be easily
transformed into Hankel's representation of the Gamma-function and together with 
$y_n\sim\frac{b\sqrt{\rho}}{2\sqrt{\pi}}\rho^{-n}n^{-\nicefrac{3}{2}}$ (cf. Equation
\eqref{eq:yn}) and $\alpha=\kappa b\sqrt{\rho}$ we obtain 
\begin{align*}
 \mathbb{E}\left(X_n^{(d_1)}(k)\cdot
X_n^{(d_2)}(k))\right)&=C_{d_1}C_{d_2}\rho^{d_1+d_2}\frac{2}{b^2\rho}n\left(e^{-\frac{\kappa^2b^2\rho}{4}}+e^{-\kappa^2b^2\rho}\right)+O(\sqrt
n).
\end{align*}
From this the representation of the covariance given in Proposition~\ref{prop:covariance} easily 
follows.

\begin{figure}[ht]
\centering
\includegraphics[width=0.8\textwidth]{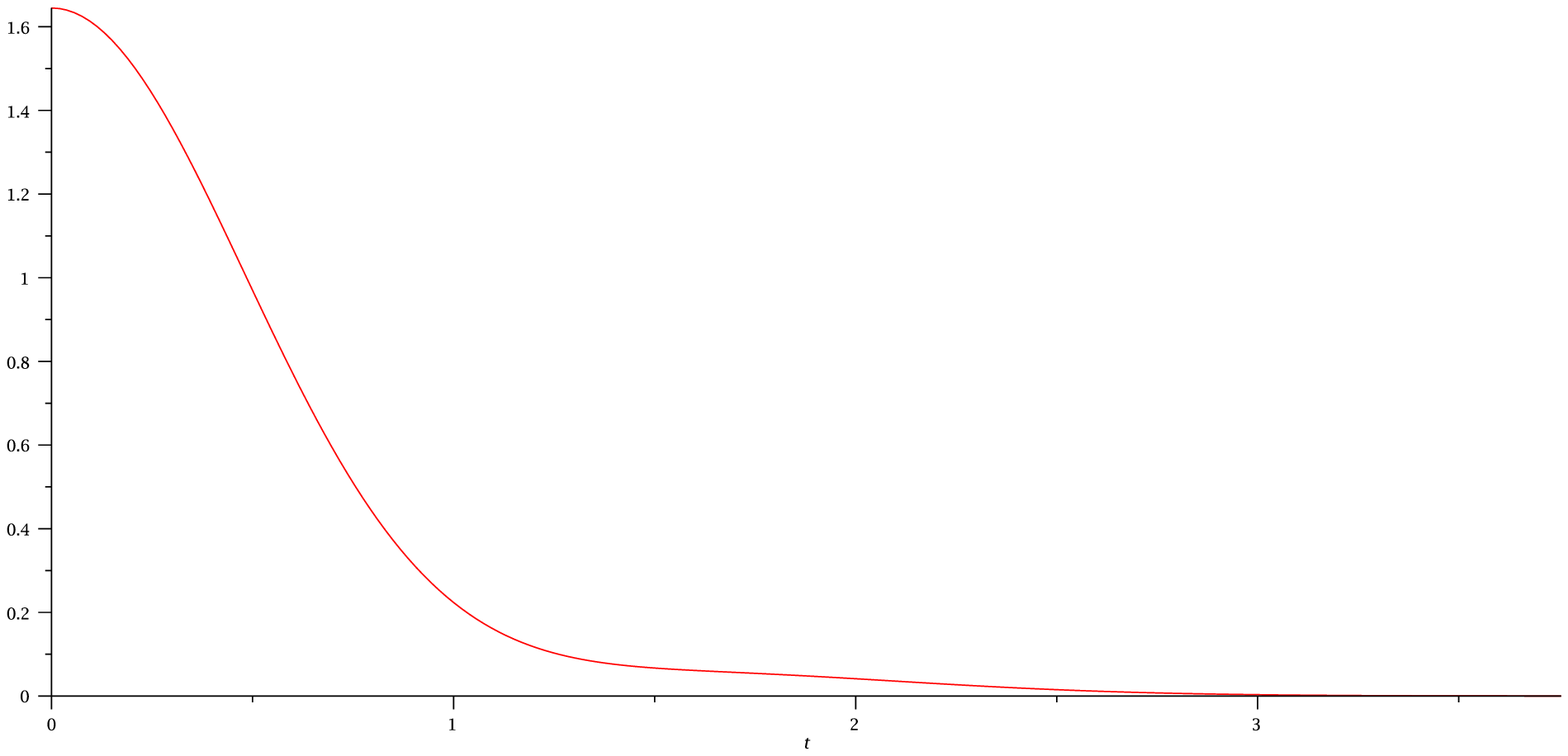}
\caption{The covariance for $\kappa\in[0,\frac{1}{\sqrt{n}}\left(\mathbb{E}(H_n)+3\sqrt{\mathbb{V}\mathrm{ar}(H_n)}\right)]$}\label{fig:covariance}
\end{figure}

\subsection{The correlation coefficient}
To obtain more information on the correlation of two degrees $d_1$ and $d_2$ on the same level $k=\kappa\sqrt{n}$, we compute the correlation coefficient, given by
\[\mathrm{Cor}\left(X_n^{(d_1)}(k), X_n^{(d_2)}(k))\right)=\frac{\mathbb{C}\mathrm{ov}(X_n^{(d_1)}(k),X_n^{(d_2)}(k))}{\sqrt{\mathbb{V}\mathrm{ar}(X_n^{(d_1)}(k))}\sqrt{\mathbb{V}\mathrm{ar}(X_n^{(d_2)}(k))}}.\]

For the computation, it remains to compute the variance $\mathbb{V}\mathrm{ar}(X_n^{(d_1)}(k))$, given by \[\mathbb{V}\mathrm{ar}(X_n^{(d_1)}(k))=\mathbb{E}\left((X_n^{(d_1)}(k))^2\right)-\left(\mathbb{E}(X_n^{(d_1)}(k))\right)^2.\] 
We need to determine $\mathbb{E}\left((X_n^{(d_1)}(k))^2\right)$, which can be done very similarly to the previous part.
\[\mathbb{E}\left((X_n^{(d_1)}(k))^2\right)=\frac{1}{y_n}[x^n]\left[\frac{\partial}{\partial u}\left(u\diff{u}y_k(x,u,1)\right)\right]_{u=1},\]

\begin{prop}\label{prop:variance}
 The Variance $\mathbb{V}\mathrm{ar}(X_n^{(d_1)}(k))$ of the random variable $X_n^{(d_1)}(k)$
counting vertices of degree $d_1$, with $d_1$ fixed, at level $k=\kappa\sqrt{n}$ in a random P\'olya
tree of size $n$ is asymptotically given by
\begin{equation}\label{eq:variance}
 \mathbb{V}\mathrm{ar}(X_n^{(d_1)}(k))=C_{d_1}C_{d_2}\rho^{2d_1}n\left(\frac{2}{b^2\rho}\left(e^{-\frac{\kappa^2b^2\rho}{4}}+e^{-\kappa^2b^2\rho}\right)-\kappa^2e^{-\frac{\kappa^2b^2\rho}{2}}\right)+O(\sqrt
n\,),
\end{equation}
as $n$ tends to infinity. 
\end{prop}

We proceed analogously to the computation of the variance, and obtain the following auxiliary result. 
\begin{lem}\label{lem:varrep}
 There exist constants $\epsilon$ and $\theta$ such that for $z\in\Delta(\eta,\theta)$
\[\tilde{\gamma}_k^{(d_1[2])}(x)=(C^{(d_1)}(x))^2y(x)^{k+2d_1}\frac{1-y(x)^k}{1-y(x)}+C^{(d_1)}(x)y(x)^{k+d_1},\]
where $C^{(d_1)}(x)$ is given in Lemma~\ref{l:g}.
\end{lem}
\begin{proof}
 The proof of this lemma is analogous to the proof of Lemma~\ref{lem:covrep}, derivating recurrence \eqref{eq:recgamma} a second time. The additional summand $C^{(d_1)}(x)y(x)^{k+d_1}$ origins in derivating twice with respect to the same variable $u$.
\end{proof}

Note that the additional summand $C^{(d_1)}(x)y(x)^{k+d_1}$ in Lemma~\ref{lem:varrep}, where $\tilde{\gamma}_k^{(d_1[2])}(x)$ and $\tilde{\gamma}_k^{(d_1,d_2)}(x)$ differ from each other, is equal to the expexted value $\mathbb{E}\left(X_n^{(d_1)}(k)\right)$ when extracting coefficients $\frac{1}{y_n}[x^n]C^{(d_1)}(x)y(x)^{k+d_1}$. As this is of order $\sqrt{n}$, while the coefficient of the other terms will be of order $n$, this term is negligible, and we obtain
\begin{equation}
 \mathbb{E}\left((X_n^{(d_1)}(k))^2\right)=C_{d_1}C_{d_2}\rho^{2d_1}\frac{2}{b^2\rho}n\left(e^{-\frac{\kappa^2b^2\rho}{4}}+e^{-\kappa^2b^2\rho}\right)
\end{equation}
by using Cauchy's formula and the integration contour $\delta$ given in \eqref{eq:delta}. 
Applying the known estimate for $\mathbb{E}(X_n^{(d_1)}(k))$ we obtain the representation given in Proposition~\ref{prop:variance}, and with Proposition~\ref{prop:covariance} the result given in Theorem~\ref{thm:cor} follows immediately.

\end{document}